\documentclass[a4paper]{amsart}
\usepackage[latin1]{inputenc}
\usepackage{amssymb}
\usepackage{amsmath}
\usepackage{mathrsfs}
\usepackage{eufrak}
\usepackage{amsthm}
\usepackage{amsfonts}
\usepackage{textcomp}
\usepackage{graphicx}
\usepackage[pdftex]{color}
\usepackage{paralist}
\usepackage[shortlabels]{enumitem}
\usepackage[hyperpageref]{backref}
\usepackage[pagebackref]{hyperref}
\renewcommand*{\backref}[1]{}  
   \renewcommand*{\backrefalt}[4]{
      \ifcase #1 
         Not cited.
      \or
         Cited on page #2.
      \else
         Cited on pages #2.
      \fi}
\usepackage{comment}
\usepackage[arrow, matrix, curve]{xy}

\usepackage{adjustbox}

\usepackage{tikz} \usetikzlibrary{matrix,patterns,shapes,decorations.pathmorphing,decorations.pathreplacing,calc,arrows} 
\usepackage{tikz-cd} 
\usepackage{tkz-graph}

\newcommand{\mm}{{\bullet}}
\newcommand{\ff}{{\color{cyan}\bullet}}
\newcommand{\ee}{{\color{black}\circ}}

\newtheorem*{corollary*}{Corollary}
\newtheorem{theorem}{Theorem}[section]

\newtheorem{corollary}[theorem]{Corollary}
\newtheorem{lemma}[theorem]{Lemma}
\newtheorem{proposition}[theorem]{Proposition}
\newtheorem{question}[theorem]{Question}

\newtheorem*{claim*}{Claim}

\theoremstyle{definition}
\newtheorem{definition}[theorem]{Definition}

\newtheorem{main conjecture}[theorem]{Main Conjecture}
\newtheorem*{theorem }{Theorem}

\newtheorem{example}[theorem]{Example}

\theoremstyle{remark}

\numberwithin{equation}{section}

\makeatletter
\renewcommand*\env@matrix[1][\
arraystretch]{%
  \edef\arraystretch{#1}%
  \hskip -\arraycolsep
  \let\@ifnextchar\new@ifnextchar
  \array{*\c@MaxMatrixCols c}}
\makeatother

\renewcommand{\mod}{\operatorname{mod}}
\newcommand{\proj}{\operatorname{proj}}
\newcommand{\inj}{\operatorname{inj}}

\newcommand{\Fac}{\operatorname{\mathrm{Fac}}}
\newcommand{\Ext}{\operatorname{Ext}}
\newcommand{\End}{\operatorname{End}}
\newcommand{\CM}{\operatorname{CM}}
\newcommand{\fidim}{\operatorname{fidim}}

\newcommand{\fpdim}{\operatorname{fpdim}}
\newcommand{\Dom}{\operatorname{dom}}

\newcommand{\Hom}{\operatorname{Hom}}
\newcommand{\add}{\operatorname{\mathrm{add}}}
\renewcommand{\top}{\operatorname{\mathrm{top}}}
\newcommand{\rad}{\operatorname{\mathrm{rad}}}
\newcommand{\soc}{\operatorname{\mathrm{soc}}}

\newcommand{\ind}{\operatorname{\mathrm{ind}}}
\newcommand{\op}{\operatorname{\mathrm{op}}}
\renewcommand{\c}{\operatorname{\mathrm{c}}}

\renewcommand{\mod}{\operatorname{mod}}

\newcommand{\F}{\mathcal{F}}
\newcommand{\M}{\mathcal{M}}
\renewcommand{\P}{\mathcal{P}}
\newcommand{\I}{\mathcal{I}}
\newcommand{\T}{\mathcal{T}}

\newcommand{\RHom}{\operatorname{\mathsf{R}Hom}}

\newcommand{\f}{{\rm f}}
\newcommand{\m}{{\rm m}}
\newcommand{\CC}{\mathcal{C}}

\newcommand{\Db}{\mathrm{D}^{\mathrm{b}}}

\newcommand{\domdim}{\operatorname{domdim}}

\newcommand{\idim}{\operatorname{idim}}

\newcommand{\pdim}{\operatorname{pdim}}
\newcommand{\gldim}{\operatorname{gldim}}

\newcommand{\cotilt}{\operatorname{cotilt}}

\newcommand{\sm}[1]{\begin{smallmatrix}#1\end{smallmatrix}}
\newcommand{\ssp}{\;\;}
\newcommand{\gr}[1]{{\color[rgb]{0.85,0.85,0.85}#1}}

\begin{document}

\title[{Auslander-Reiten's Cohen-Macaulay algebras and contracted preprojective algebras}]{Auslander-Reiten's Cohen-Macaulay algebras and\\ contracted preprojective algebras}
\date{\today}



\author[Chan]{Aaron Chan}
\address[Chan]{Graduate School of Mathematics, Nagoya University, Furocho, Chikusaku, Nagoya 464-8602, Japan}
\email{aaron.kychan@gmail.com}

\author[Iyama]{Osamu Iyama}
\address[Iyama]{Graduate School of Mathematical Sciences, University of Tokyo, 3-8-1 Komaba Meguro-ku Tokyo 153-8914, Japan}
\email{iyama@ms.u-tokyo.ac.jp}

\author[Marczinzik]{Ren\'{e} Marczinzik}
\address[Marczinzik]{Mathematical Institute of the University of Bonn, Endenicher Allee 60, 53115 Bonn, Germany}
\email{marczire@math.uni-bonn.de}

\dedicatory{Dedicated to Jun-ichi Miyachi for his 65th Birthday}

\begin{abstract}
Auslander and Reiten called a finite dimensional algebra $A$ over a field \emph{Cohen-Macaulay} if there is an $A$-bimodule $W$ which gives an equivalence between the category of finitely generated $A$-modules of finite projective dimension and the category of finitely generated $A$-modules of finite injective dimension.
For example, Iwanaga-Gorenstein algebras and algebras with finitistic dimension zero on both sides are Cohen-Macaulay, and tensor products of Cohen-Macaulay algebras are again Cohen-Macaulay. They seem to be all of the known examples of Cohen-Macaulay algebras.

In this paper, we give the first non-trivial class of Cohen-Macaulay algebras by showing that all contracted preprojective algebras of Dynkin type are Cohen-Macaulay. As a consequence, for each simple singularity $R$ and a maximal Cohen-Macaulay $R$-module $M$, the stable endomorphism algebra $\underline{\End}_R(M)$ is Cohen-Macaulay. We also give a negative answer to a question of Auslander-Reiten asking whether the category $\CM A$ of Cohen-Macaulay $A$-modules coincides with the category of $d$-th syzygies, where $d\ge1$ is the injective dimension of $W$.
In fact, if $A$ is a Cohen-Macaulay algebra that is additionally $d$-Gorenstein in the sense of Auslander, then $\CM A$ always coincides with the category of $d$-th syzygies.

\end{abstract}

\maketitle
\tableofcontents
 
 \section*{Introduction}
Let $R$ be a commutative Noetherian local ring of Krull dimension $d$. Then $M\in\mod R$ is called \emph{Cohen-Macaulay} (or \emph{CM}) if the depth of $M$ coincides with the dimension of $M$, and the ring $R$ is called \emph{Cohen-Macaulay} (or \emph{CM}) if $R$ as an $R$-module is CM. In this case, a \emph{canonical} $R$-module is a CM $R$-module $\omega$ with $\dim \omega=d$ and finite injective dimension such that $\Ext^d_R(k,\omega)\simeq k$ for the residue field $k$ of $R$. It is also called a \emph{dualizing} $R$-module since it is a dualizing complex concentrated in degree zero and hence gives a duality $\RHom_R(-,\omega):\Db(\mod R)\simeq\Db(\mod R)$ \cite{H}.
A Cohen-Macaulay local ring $R$ admits a dualizing module if and only if it is a homomorphic image of a Gorenstein ring \cite[3.3.6]{BH}. In particular, each complete local Cohen-Macaulay ring admits a dualizing $R$-module. One of the important properties of a dualizing $R$-module is that it induces quasi-inverse equivalences
\[-\otimes_R\omega:\P^{<\infty}(R)\to\I^{<\infty}(R)\ \mbox{ and }\ \Hom_R(\omega,-) : \I^{< \infty}(R) \rightarrow \P^{< \infty}(R),\]
where $\P^{< \infty}(R)$ is the full subcategory of $\mod R$ consisting of modules with finite projective dimension, and $\I^{< \infty}(R)$ is the full subcategory of $\mod R$ with finite injective dimension \cite{Sh}. These equivalences have been extended to the level of derived categories and generalized to commutative rings with dualizing complexes \cite{AF}, see also \cite{IK}.

In \cite{AR2}, Auslander and Reiten introduced a non-commutative generalisation of Cohen-Macaulay rings. Their definition, generalized in \cite{BR}, is as follows:

\begin{definition}\label{define CM}
We call a (not necessarily commutative) Noetherian ring $A$ \emph{Cohen-Macaulay} if there is an $A$-bimodule $W$ which gives quasi-inverse equivalences
\[-\otimes_AW:\P^{<\infty}(A)\to\I^{<\infty}(A)\ \mbox{ and }\ \Hom_A(W,-) : \I^{< \infty}(A) \rightarrow \P^{< \infty}(A),\]
where $\P^{< \infty}(A)$ is the full subcategory of $\mod A$ consisting of modules with finite projective dimension, and $\I^{< \infty}(A)$ is the full subcategory of $\mod A$ with finite injective dimension. We call $W$ in this case a \emph{dualizing} module.
\end{definition}

For example, Iwanaga-Gorenstein rings \cite{EJ}, which are Noetherian rings $A$ such that the injective dimensions $\idim_AA$ and $\idim_{A^{\op}}A$ are finite,
are precisely Cohen-Macaulay algebras $A$ such that the $A$-bimodules $A$ give dualizing modules.
If $A$ is a commutative Noetherian complete local ring, then the notion of Cohen-Macaulay ring above coincides with the classical one 
thanks to the validity of the famous Bass conjecture (that is, a commutative local ring is CM if and only if there is a finitely generated injective module, see \cite{PS} and \cite{R}).
In \cite{AR2}, several classical results about Iwanaga-Gorenstein rings are generalized for Cohen-Macaulay rings.
We refer to \cite{AR,B,BFS,BR,BST,GN,Ni} for more results on Cohen-Macaulay rings in the sense of Definition \ref{define CM}.

From now on, we denote by $A$ a finite dimensional algebra over a field $k$. 
Recall that the \emph{finitistic projective/injective dimensions} of $A$ are defined by
\[\fpdim A:=\sup\{\pdim X\mid X\in\P^{<\infty}(A)\}\ \mbox{ and }\ \fidim A:=\sup\{\idim X\mid X\in\I^{<\infty}(A)\}.\]
Since $\fpdim A=\fidim A^{\op}$ clearly holds, we only consider $\fidim$ throughout this paper.
One of the nice properties of CM algebras $A$ is 
\begin{equation}\label{4 dimension}
\fidim A=\idim_AW=\idim_{A^{\op}}W=\fidim A^{\op},
\end{equation}
see \cite[Proposition 1.6]{AR2}. Hence the famous finitistic dimension conjecture holds true for CM algebras, and $\fidim A$ gives their important homological invariant. 

Auslander and Reiten characterised dualizing modules in terms of tilting theory: 
Recall that the set $\cotilt A$ of additive equivalence classes of cotilting $A$-modules has a natural partial order given by $T\ge U\Leftrightarrow{\rm Ext}^i_A(T,U)=0$ for all $i\ge1$. Then $\cotilt A$ has a minimal element $DA$, while we call a maximal element of $\cotilt A$ \emph{$\Ext$-maximal}, which does not necessarily exist (see \cite{HU1,HU2,IZ} for more details).

\begin{proposition}\cite[1.3]{AR2}\label{dualizing = 2sided Extmax+DCP}
An $A$-module $W$ is a dualizing $A$-module if and only if the following conditions are satisfied.
\begin{enumerate}[\rm(i)]
\item $W$ is an $\Ext$-maximal cotilting $A$-module.
\item $W$ is an $\Ext$-maximal cotilting $\End_A(W)^{\op}$-module.
\item There is an $k$-algebra isomorphism $A\to\End_A(W)$.
\end{enumerate}
\end{proposition}


All known examples of CM algebras seem to be one of the following.
\begin{enumerate}
\item Iwanaga-Gorenstein algebras,
\item algebras $A$ with $\fidim A=\fidim A^{\op}=0$,
\item tensor products of algebras in (1) and (2).
\end{enumerate}
In this article we give the first new examples of CM algebras that are not contained in the list above, since those new example are in general non-Iwanaga-Gorenstein and can have positive dominant dimension, while the algebras mentioned in (1)-(3) are always Iwanaga-Gorenstein or have dominant dimension zero.

The preprojective algebras of quivers are important algebras that appear in various areas of
mathematics, e.g.\ Cohen--Macaulay modules \cite{Aus86, GL91}, Kleinian singularities \cite{Cra00}, cluster algebras \cite{GLS13}, quantum groups \cite{KS97, Lus91}, and quiver varieties \cite{Nak94}.
For a graph $\Delta$, we fix an orientation to get a quiver $Q$. Then define the double $\overline{Q}$  by adding a new arrow $a^*:j\to i$ for each arrow $a:i\to j$ in $Q$. The \emph{preprojective algebra} of $\Delta$ is defined by
\[\Pi=\Pi(\Delta):=k\overline{Q}/\langle\sum_{a\in Q_1}(aa^*-a^*a)\rangle.\]
Clearly $\Pi$ does not depend on a choice of the orientation of $\Delta$.
It is well-known that $\dim_k\Pi$ is finite if and only if $\Delta$ is a Dynkin diagram.
A \emph{contracted preprojective algebra} is an algebra of the form $e\Pi e$, where $\Pi$ is a preprojective algebra and $e\in\Pi$ is an idempotent \cite{IW}. It is called of \emph{Dynkin type} if $\Delta$ is a Dynkin diagram. In \cite{IW}, contracted preprojective algebras play a key role to study non-commutative crepant resolutions of cDV singularities.
Now we able to state a main result of this paper.

\begin{theorem}[Theorem \ref{main ePie 1}]\label{contracted}
Each contracted preprojective algebra $A$ of Dynkin type is a Cohen-Macaulay algebra such that $\fidim A$ is either 0 or 2. 
\end{theorem}

We also calculate explicitly homological dimensions of $A$ in Theorem \ref{dimension for type A}, that is, $\fidim A$, $\domdim A$, $\idim A$ and $\gldim A$, see also Example \ref{tableexample}.
To prove Theorem \ref{contracted}, we give a general Theorem \ref{main triangulated G} which gives a Cohen-Macaulay algebra from a triple $(\T,\M,\F)$ of a triangulated category $\T$ and its subcategory $\M\supset\F$ satisfying a certain set of axioms, see Definition \ref{axiom T,M,F}.

For a CM algebra $A$ with dualizing module $W$,
the category of \emph{Cohen-Macaulay} (or \emph{CM}) $A$-modules is defined as
\[\CM A:= \{ X \in \mod A\mid \forall i>0,\ \Ext_A^i(X,W)=0\}.\] 
If $A$ is Iwanaga-Gorenstein with $\idim A=d$, then the equality $\CM A=\Omega^d(\mod A)$ holds \cite[Proposition 3.1]{AR2}. If $A$ is an algebra with $\fidim A=\fidim A^{\op}=0$, then $W=DA$ and hence $\CM A=\mod A$ trivially. Motivated by these observations, Auslander and Reiten posed the following question that was stated after \cite[Proposition 3.1]{AR2}.

\begin{question}\label{AR conjecture} (Auslander-Reiten)
If $A$ is Cohen-Macaulay of finitistic dimension $d\ge 1$ with $\CM A=\Omega^d(\mod A)$, is $A$ Iwanaga-Gorenstein?
\end{question} 

As the second main result of this paper, we give a family of counter examples to Question \ref{AR conjecture}. In fact, we prove that some of the contracted preprojective algebras $A$ (which are Cohen-Macaulay by Theorem \ref{contracted}) are not Iwanaga-Gorenstein but have finitistic dimension two and satisfy $\CM A=\Omega^2(\mod A)$.

This equality is a special case of a more general result, which gives a sufficient condition for a Cohen-Macaulay algebra $A$ of finitistic dimension $d$ to satisfy $\CM A=\Omega^d(\mod A)$.
Somehow surprisingly, the condition is given by the classical $n$-Gorenstein condition due to Auslander and Fossum-Griffith-Reiten \cite{FGR} (see also \cite{AR3,AR4,Hu,HI}). For $X\in\mod A$, we denote by
\begin{eqnarray*}
&\cdots \rightarrow P_2(X) \rightarrow P_1(X) \rightarrow P_0(X)\rightarrow X \rightarrow 0&\ \mbox{ and }\\
&0 \rightarrow X \rightarrow I^0(X) \rightarrow I^1(X) \rightarrow I^2(X) \rightarrow\cdots&
\end{eqnarray*}
the minimal projective resolution and the minimal injective coresolution of $X$ respectively.
Recall that an algebra $A$ is called \emph{$n$-Gorenstein} if 
$\pdim I^i(A) \leq i$ (respectively, \emph{quasi $n$-Gorenstein} if $\pdim I^i(A)\leq i+1$) holds for each $0\le i\le n-1$. Note that $A$ is $n$-Gorenstein if and only if so is $A^{\op}$ \cite{FGR}.

\begin{theorem}[Corollary \ref{syzygiesACM2}]\label{syzygiesACM2 intro}
Let $A$ be a Cohen-Macaulay algebra with $\fidim A=d$. If $A$ is $d$-Gorenstein (or more generally, $A^{\op}$ is quasi $d$-Gorenstein), then $\CM A=\Omega^d(\mod A)$. 
\end{theorem}

The equality $\CM A=\Omega^d(\mod A)$ in Theorem \ref{syzygiesACM2 intro} should be regarded as an analogue of a fundamental result in Cohen-Macaulay representation theory of isolated singularities, where both $\CM A$ and $\Omega^d(\mod A)$ coincide with the category of $d$-torsionfree modules \cite{Aus1} (see \cite[1.3.1(2)]{I},  \cite[Corollary A.15]{LW}).

To answer the question of Auslander and Reiten, we will classify when the stable endomorphism rings as in the previous theorem are at least 1-Gorenstein, which turns out to be equivalent to having dominant dimension at least two. The answer is surprisingly simple:

\begin{theorem}
Let $A$ be a contracted preprojective algebra of Dynkin type, and $W$ the dualizing $A$-module.
\begin{enumerate}[\rm(1)]
\item If $\domdim A\ge2$, then $\CM A=\Omega^2(\mod A)$ holds. 
\item There are infinitely many contracted preprojective algebras $A$ satisfying $\domdim A\ge2$ and $\idim A=\infty$ (see Proposition \ref{family of counter example}). All of them are counter examples to Question \ref{AR conjecture}.
\end{enumerate}
\end{theorem}

\medskip
We end the introduction by giving a remarkable application of our results.
Let $k$ be an algebraically closed field of characteristic 0. Then
a \emph{simple singularity} over $k$ is a hypersurface
$k[[x_0,x_1,...,x_d]]/(f_\Delta^{d})$ classified by Dynkin diagrams $\Delta$, where $f_\Delta^{d}$ is one of the following types:
\begin{enumerate}
\item $f_{A_n}^{d}=x_0^2+x_1^{n+1}+x_2^2+\cdots+x_d^2, \ n \geq 1.$
\item $f_{D_n}^{d}=x_0^2 x_1+x_1^{n-1}+x_2^2+\cdots+x_d^2, n \geq 4.$
\item $f_{E_6}^{d}=x_0^3+x_1^{4}+x_2^2+\cdots+x_d^2.$
\item $f_{E_7}^{d}=x_0^3+x_0x_1^{3}+x_2^2+\cdots+x_d^2.$
\item $f_{E_8}^{d}=x_0^2+x_1^{5}+x_2^2+\cdots+x_d^2.$
\end{enumerate}
The simple singularities are characterised as hypersurface singularities of finite deformation type, we refer to \cite{Ar1,Ar2} for more details.
On the other hand, recall that a complete local Cohen-Macaulay ring $R$ is called of \emph{finite Cohen-Macaulay type} if there are only finitely many indecomposable maximal Cohen-Macaulay modules up to isomorphism.
The famous theorem by Buchweitz, Greuel, Schreyer and Kn\"orrer \cite{BGS,K} gave a representation theoretic characterisation of the simple singularities:
A complete local Gorenstein ring $R$ of Krull dimension $d$ over an algebraically closed field of characteristic 0 is of finite Cohen-Macaulay type if and only if $R$ is a simple  singularity. We refer to \cite{Y,LW} for more details.

Nowadays it is well-known in representation theory that those simple singularities of dimension 2 have preprojective algebras of extended Dynkin type as their non-commutative crepant resolutions. In particular, contracted preprojective algebras of Dynkin type are precisely the stable endomorphism rings of maximal Cohen-Macaulay modules over simple singularities of dimension $2$. Therefore Theorem \ref{contracted}, together with Kn\"orrer periodicity and the Calabi-Yau property of $\underline{\CM}R$ gives the following remarkable result.

\begin{theorem}[Theorem \ref{ARdominantdimension2criterion}]
Let $R=k[[x_0,x_1,...,x_d]]/(f_\Delta^{d})$ be a simple singularity with an algebraically closed field $k$ of characteristic 0. For each maximal Cohen-Macaulay $R$-module $M$, the stable endomorphism ring of $\underline{\End}_R(M)$ is a Cohen-Macaulay algebra in the sense of Definition \ref{define CM}.
\end{theorem}

\section*{Acknowledgements} 
This project profited from the use of the GAP-package \cite{QPA}. A part of this work was done when the authors attended the workshop ``Representation Theory of Quivers and Finite Dimensional Algebras'' in Mathematisches Forschungsinstitut Oberwolfach.  AC thanks also  Universit\"{a}t Stuttgart for their hospitality while part of this work was done.
OI thanks Michael Wemyss for ongoing joint work on contracted preprojective algebras \cite{IW}.
AC is supported by JSPS Grant-in-Aid for Research Activity Start-up program 19K23401 and Scientific Research (C) 24K06666.  OI is supported by JSPS Grant-in-Aid for Scientific Research (B) 22H01113 and (C) 18K3209.  RM was supported by the DFG with the project number 428999796.

\section{Preliminaries}

Throughout this section, let $A$ be a finite dimensional algebra over a field $k$.
All modules are right modules.
The composition of morphisms $f:X\to Y$ and $g:Y\to Z$ is denoted by $g\circ f:X\to Z$. Thus $X$ is an $\End_A(X)^{\op}$-module.
The composition of arrows $a:i\to j$ and $b:j\to k$ is denoted by $ab$.
For the basics on representation theory and homological algebra of finite dimensional algebras, we refer for example to \cite{ARS,ASS,SkoYam}. $D=\Hom_k(-,k)$ denotes the natural duality of $\mod A$ for a finite dimensional $k$-algebra $A$.

We will now recall the definition of a few classes of algebras.
Their relations can be summarised in the diagram below, where we abbreviate Gorenstein to `$\text{Gor.}$' for space-saving purposes.
\[
\xymatrix{
\text{\phantom{aa}selfinjective\phantom{aa}}\ar@{=>}[r]\ar@{=>}[d] & \text{Auslander-Gor.}\ar@{=>}[r]\ar@{=>}[dr]& \text{Iwanaga-Gor.}\ar@{=>}[r] & \text{Cohen-Macaulay}\ar@{=>}[r] & \fidim A<\infty\\
\domdim A=\infty \ar@{=>}[r] & \domdim A\ge n \ar@{=>}[r] & n\text{-Gorenstein} &&}
\]

An algebra $A$ is called \emph{Iwanaga-Gorenstein} if $\idim A_A$ and $\idim_AA$ are finite. In this case, we clearly have $\idim A_A=\idim_AA$,
which is then called the \emph{selfinjective dimension} of $A$. 
On the other hand, for $X\in\mod A$, we denote by
$$0 \rightarrow X \rightarrow I^0(X) \rightarrow I^1(X) \rightarrow \cdots $$
a minimal injective resolution of $X$. 
The \emph{dominant dimension} of $X\in\mod A$ is defined as the minimal $n$ such that $I^n(X)$ is not projective or as infinite if no such $n$ exists. 
The dominant dimension of an algebra $A$ is defined as the dominant dimension of the regular representation $A\in\mod A$.
For example, selfinjective algebras have infinite dominant dimension.

We call $A$ \emph{$n$-Gorenstein} (respectively, \emph{quasi $n$-Gorenstein}) if $\pdim I^i(A) \leq i$ (respectively, $\pdim I^i(A) \leq i+1$) holds for all $0\le i\le n-1$.   It is well known that $A$ is $n$-Gorenstein if and only if $A^{\op}$ is $n$-Gorenstein. We call $A$ \emph{Auslander-Gorenstein} if $A$ is Iwanaga-Gorenstein and $n$-Gorenstein for all $n\ge 1$.

For a module $M$ define $\add M$ as the full subcategory of modules that are direct summands of $M^n$ for some $n$.
For a subcategory $\CC$, define $\widehat{\mathcal{C}}$ as the full subcategory of modules $N$ such that there is an exact sequence $0 \rightarrow X_n \rightarrow\cdots\rightarrow X_0 \rightarrow N \rightarrow 0$ with $X_i \in \mathcal{C}$.
Define ${}^\perp \CC$ (resp. ${}^{\perp n}\CC$) as the full subcategory of modules $X$ satisfying $\Ext_A^{i}(X,C)=0$ for all $C\in \CC$ and all $i>0$ (resp. all $0<i\le n$).
We also use the shorthand notation ${}^\perp M:={}^\perp (\add M)$ and ${}^{\perp n}M:={}^{\perp n}(\add M)$. 

Denote by $\I^{< \infty}(A)$ the full subcategory of $A$-modules having finite injective dimension (i.e. $\I^{<\infty}(A)=\widehat{\add DA}$).
For each $n\ge 0$, denote by $\I^{\leq n}(A)$ the full subcategory of $A$-modules having finite injective dimension at most $n$.
Define $\P^{< \infty}(A)$ and $\P^{\le n}(A)$ similarly where we replace injective by projective.
Clearly an algebra $A$ is Iwanaga-Gorenstein if and only if $\I^{< \infty}(A)=\P^{< \infty}(A)$.
Recall that the \emph{finitistic injective dimension} and the \emph{finitistic projective dimension} of $A$ is defined as
\begin{align*}
\fidim A&:=\sup\{\idim X\mid X\in\I^{<\infty}(A) \},\\
\fpdim A&:=\sup\{\pdim X\mid X\in\P^{<\infty}(A) \}.
\end{align*}
Clearly $\fpdim A=\fidim A^{\op}$ holds, but $\fidim A$ and $\fpdim A$ are different in general.
For simplicity, we will often simply speak of the \emph{finitistic dimension} of an algebra and mean the finitistic injective dimension since this is more convenient in the study of Cohen-Macaulay algebras, see \eqref{4 dimension}.

For $X\in\mod A$, we denote by $\Omega^n(X)$ the $n$-th syzygy of $X$ given by a minimal projective resolution of $X$. We consider the full subcategory of $\mod A$ defined by
\[\Omega^n(\mod A):=\add\{A,\Omega^n(X)\mid X\in\mod A\}.\]
Note that in \cite{AR3}, $\Omega^n(\mod A)$ is defined without taking the additive closure. The two definitions coincide if $A$ is $n$-Gorenstein, see \cite[Proposition 3.5]{AR3}.

A module $T$ is called a \emph{cotilting module} if it has finite injective dimension, $\Ext_A^i(T,T)=0$ for all $i>0$ and $DA \in \widehat{\add T}$.
Let $\cotilt_nA$ be the set of additive equivalence classes of cotilting $A$-modules of injective dimension at most $n$, and $\cotilt A:=\bigcup_{n\ge0}\cotilt_nA$. The following Auslander-Reiten correspondence is important.

\begin{theorem}[Auslander-Reiten correspondence]\cite[Theorem 5.5]{AR}\label{AR cotilt corresp}
Let $A$ be an algebra and $n\ge0$. There is a bijection between the following sets.
\begin{enumerate}[\rm(1)]
\item $\cotilt_nA$.
\item the set of contravariantly finite resolving subcategories $\mathcal{C}$ of $\mod A$ containing $\Omega^n(\mod A)$.
\item the set of covariantly finite coresolving subcategories $\mathcal{C}$ of $\mod A$ contained in $\I^{\le n}(A)$.
\end{enumerate}
This bijections from (1) to (2) and (1) to (3) are given by $T\mapsto{}^\perp T$ and $T\mapsto({}^\perp T)^\perp=\widehat{\add T}$.
\end{theorem}

Note that the subcategories satisfying (2) above form a poset under inclusion.
On the other hand, the set $\cotilt A$ has a natural partial order given by
\[T\ge U :\Leftrightarrow \Ext^i_A(T,U)=0\ \mbox{ for all }\ i\ge1.\]
Thanks to works by Happel-Unger \cite{HU1,HU2}, it is known that the Auslander-Reiten correspondence from (1) to (2) is a poset anti-isomorphism. In particular, the following conditions for $T\in\cotilt A$ are equivalent (see \cite[Theorem 3.1]{IZ}).
\begin{enumerate}[$\bullet$]
\item $T$ is a maximal element of $\cotilt A$ (respectively, $\cotilt_nA$).
\item $T$ is the maximum element of $\cotilt A$ (respectively, $\cotilt_nA$). 
\item $\widehat{\add T}=\I^{< \infty}(A)$ (respectively, $\widehat{\add T}=\I^{\leq n}(A)$).
\end{enumerate}
Furthermore, such a $T$ exists if and only if $\I^{<\infty}(A)$ is covariantly finite (respectively, $\I^{\leq n}(A)$ is covariantly finite).
In this case, we call $T$ \emph{$\Ext$-maximal} (respectively, \emph{$n$-$\Ext$-maximal}). Note that $\Ext$-maximal cotilting modules are also called \emph{strong} cotilting modules \cite{AR,AR2}.


Let $U$ be a cotilting $A$-module with decomposition $U=V\oplus W$. Take an exact sequence $0\to W^*\to V'\xrightarrow{f} W$ with right $(\add V)$-approximation $f$. If $f$ is surjective, then
\begin{equation}\label{mu^+_V}
\mu^+_V(U):=V\oplus W^*
\end{equation}
is a cotilting $A$-module satisfying $\mu^+_V(U)>U$, called the \emph{mutation} of $U$.



\begin{proposition} \label{CMcharapropo}
Let $A$ be an algebra.
For a basic cotilting $A$-module $W$, the following conditions are equivalent.
\begin{enumerate}[\rm(i)]
\item $W$ is $\Ext$-maximal.
\item For $B:=\End_A(W)$, every simple $B^{\op}$-module is a direct summand of $\top({}_BW)$.
\item For each indecomposable direct summand $X$ of $W$, we have $X\notin\Fac(W/X)$.
\end{enumerate}
\end{proposition}

\begin{proof}
(i)$\Leftrightarrow$(ii) is dual to \cite[Proposition 7.1]{DH}. (ii)$\Leftrightarrow$(iii) This was shown in \cite[Lemma 3.1]{HU2}.
\end{proof}

Propositions \ref{dualizing = 2sided Extmax+DCP} and \ref{CMcharapropo} give a useful criterion for a given algebra $A$ to be Cohen-Macaulay.

\begin{example}\label{eg:non-CM}
We give a few examples and non-examples of Cohen-Macaulay algebras. 
\begin{enumerate}
\item By Proposition \ref{CMcharapropo}, $\fidim A=0$ holds if and only if $\soc A_A$ contains all simple $A$-modules as a direct summand if and only if $DA$ is an Ext-maximal cotilting $A$-module.
Thus Cohen-Macaulay algebras $A$ with $\fidim A=0$ are precisely algebras $A$ such that $\soc A_A$ contains all simple $A$-modules and $\soc{}_AA$ contains all simple $A^{\op}$-modules.
Typical examples are given by local algebras and selfinjective algebras. There are many other examples, e.g.\ if $Q$ is a quiver without sinks and/or sources, then $A:=KQ/\langle\mbox{all paths of length 2}\rangle$ is a Cohen-Macaulay algebra $A$ with $\fidim A=0$.
\item Iwanaga-Gorenstein algebras are precisely Cohen-Macaulay algebras $A$ with dualizing module $A$.
\item Recall that dualizing modules can be determined by the three conditions (i)(ii)(iii) in Proposition \ref{dualizing = 2sided Extmax+DCP}. We demonstrate here that these conditions are independent of each other. For example, consider the quivers 
\[Q=\Big[\begin{tikzcd} 1 & 2 \arrow["a", from=1-1, to=1-1, loop, in=145, out=215, distance=10mm] \arrow["b", from=1-1, to=1-2] \end{tikzcd}\Big],\ \text{ and } \ \ Q'=\Big[\begin{tikzcd} 1 & 2 \arrow["b"', shift right=1, from=1-1, to=1-2] \arrow["a"', shift right=1, from=1-2, to=1-1] \end{tikzcd}\Big].\]
For $A:=kQ/\langle\mbox{all paths of length 2}\rangle$, the $A$-module $DA$ satisfies (i) and (iii), but does not satisfy (ii).
For $B:=kQ'/\langle aba\rangle$, the $B$-module $e_1B\oplus S_2$ satisfies (i) and (ii), but does not satisfy (iii).
\end{enumerate}
\end{example}

For a finite dimensional algebra $A$, we have an equivalence $\nu=-\otimes_ADA:\proj A\simeq\inj A$ called the \emph{Nakayama functor}.
We now consider the full subcategory $\Dom_nA$ of $\mod A$ consisting of modules $M$ with $\domdim M\ge n$.
For an algebra $A$ of dominant dimension at least two, take an idempotent $e\in A$ such that $D(Ae)$ is an additive generator of $\proj A\cap\inj A$. We call 
$B:=eAe$ the \emph{base algebra} of $A$. 
We will need the following results that are special cases of results in \cite{APT} in the situation of dominant dimension at least two.

\begin{proposition}\cite{APT}
\label{APTmaintheorem}
Let $A$ be an algebra of dominant dimension at least two, and
$B:=eAe$ the base algebra as above. Then we have an equivalence 
$(-)e : \Dom_2A \rightarrow \mod B$ of categories, which restricts to an equivalence between $\add I\to\inj B$.
\end{proposition}

\begin{proposition} \label{MVpropo}
Let $A$ be an algebra with dominant dimension $n \geq 0$.
\begin{enumerate}[\rm(1)]
\item\cite[Proposition 4]{MarVil} We have $\Omega^i(\mod A)=\Dom_iA$ for all $0\le i\le n$.
\item\cite[Proposition 5 and the Corollary before it]{MarVil} We have $\Dom_nA={}^{\perp}\I^{ \leq n}(A)$.
\end{enumerate} 
\end{proposition}


\section{$\Ext$-maximal cotilting modules and $d$-Gorenstein algebras}

We start with giving our main result of this section, which implies that $^{\perp}W=\Omega^n(\mod A)$ holds when $A$ is a $n$-Gorenstein algebra with an $n$-$\Ext$-maximal cotilting module $W$ of injective dimension $n$.


\begin{theorem} \label{syzygiesACM}
Let $A$ be an algebra and $n\ge0$.
If $A$ is $n$-Gorenstein (or more generally, $A^{\op}$ is quasi $n$-Gorenstein), then there exists an $n$-$\Ext$-maximal
cotilting $A$-module $W$ satisfying $^{\perp}W=\Omega^n(\mod A)$.
\end{theorem}

The proof is based on the classical Auslander-Reiten correspondence (Theorem \ref{AR cotilt corresp}).

\begin{proof}
Recall from \cite[Theorem 1.2]{AR3} that the subcategory $\Omega^n(\mod A)$ is always contravariantly finite.
On the other hand, since $A$ is $n$-Gorenstein (or more generally $A^{\op}$ is quasi $n$-Gorenstein), the subcategory $\Omega^n(\mod A)$ is closed under extensions by \cite[Theorem 2.1]{AR3}. This immediately implies that $\Omega^n(\mod A)$ is closed under kernels of epimorphisms. In fact, for an exact sequence $0\to X\to Y\to Z\to 0$ in $\mod A$ with $Y,Z\in\Omega^n(\mod A)$, take an exact sequence $0\to\Omega Z\to P\to Z\to0$ with $P\in\proj A$. Then $\Omega Z\in\Omega^n(\mod A)$, and we have an exact sequence $0\to\Omega Z\to X\oplus P\to Y\to0$. Thus we have $X\oplus P$ and hence $X$ belongs to $\Omega^n(\mod A)$.
Consequently, $\Omega^n(\mod A)$ is resolving.

Now $\Omega^n(\mod A)$ is the minimum element in the category side of Auslander-Reiten correspondence (Theorem \ref{AR cotilt corresp}).  Since this correspondence is a poset anti-isomorphism, the corresponding cotilting module $W$ is the maximum of $\cotilt_nA$, as desired.
\end{proof}

The following is an immediate consequence.

\begin{corollary} \label{syzygiesACM2}
Let $A$ be a Cohen-Macaulay algebra with $\fidim A=d$.
If $A$ is $d$-Gorenstein (or more generally, $A^{\op}$ is quasi $d$-Gorenstein),, then $\CM A=\Omega^d(\mod A)$.
\end{corollary}

For example, all the assumptions in Corollary 2.2 are satisfied by Auslander-Gorenstein algebras. 
We will give a class of examples that are non-Iwanaga-Gorenstein, but $d$-Gorenstein and Cohen-Macauley, in Section \ref{section 4}.
Namely, they are certain idempotent-truncations of preprojective algebras of Dynkin type.

As an application of Corollary \ref{syzygiesACM2}, we show the next observation, which gives a simple description of the category $\CM A$ of a special class of Cohen-Macaulay algebras. 

\begin{corollary}
Let $A$ be a Cohen-Macaulay algebra with $\fidim A=2$ and $\domdim A\ge2$, and $B$ the base algebra of $A$. Then we have an equivalence $\CM A\simeq \mod B$.
\end{corollary}

\begin{proof}
The assertion follows from $\CM A\stackrel{\ref{syzygiesACM2}}{=} \Omega^2(\mod A)\stackrel{{\rm\ref{MVpropo}(1)}}{=}\Dom_2A\stackrel{{\rm\ref{APTmaintheorem}}}{\simeq}\mod B$.
\end{proof}

We have the following explicit form of the $d$-$\Ext$-maximal module over an $d$-Gorenstein algebra. 
\begin{proposition}\cite[Corollary 3.5]{IZ}
 \label{strongcotiltforWAG}
Let $A$ be an algebra which is $n$-Gorenstein. Then $A$ has an $n$-$\Ext$-maximal cotilting $A$-module
\[W:=\Big(\bigoplus_{0\le i\le n-1}P_i(DA)\Big)\oplus\Omega^{n}(DA).\]
\end{proposition}

Now we give an equivalent condition for Cohen-Macaulay algebras of finitistic dimension $d$ to be $d$-Gorenstein in terms of the minimal injective coresolutions of the dualizing modules.
\begin{definition}
Let $A$ be an algebra, and $n\ge1$. We say that $X\in\mod A$ is \emph{$n$-Gorenstein} if
$\pdim I^i(X) \leq i$ for each $0\le i\le n-1$.
\end{definition}


The next result shows that, for an algebra $A$ with $n$-$\Ext$-maximal cotilting module $W$, the $n$-Gorensteiness of $A$ is equivalent to that of $W$.

\begin{proposition}\label{dgorensteinequalWgorenstein}
Let $A$ be an algebra and $n\ge0$ such that $A$ has an $n$-$\Ext$-maximal cotilting module $W$.
Then $A$ is $n$-Gorenstein if and only if $W$ is $n$-Gorenstein.
Moreover, in this case, we have $\pdim I^i(W) = \pdim I^i(A)$ for each $0\le i\le n-1$.
Therefore $\domdim A\ge n$ holds if and only if $\domdim W\ge n$ holds.

\end{proposition}

To prove this, we need the following observation.

\begin{lemma}\label{Miyachilemma}
Let $0 \rightarrow M_{-1} \rightarrow M_0 \rightarrow M_1 \rightarrow \cdots $ be an exact sequence. 
\begin{enumerate}[\rm(1)]
\item\cite[Lemma 1.1]{M} For $i\ge0$, let $0 \rightarrow M_i \rightarrow I_i^0 \rightarrow I_i^1 \rightarrow I_i^2 \rightarrow \cdots $ be an injective coresolution of $M_i$. Then $M_{-1}$ has an injective coresolution of the following form:
$$0 \rightarrow M_{-1} \rightarrow I_0^i \rightarrow \bigoplus_{0 \leq r \leq 1}^{}{I_r^{1-r}} \rightarrow \cdots \rightarrow \bigoplus_{0 \leq r \leq s}^{}{I_r^{s-r}} \rightarrow \cdots $$

\item Let $n\ge0$. If each $M_i$ with $0\le i\le n-1$ is $n$-Gorenstein, then so is $M_{-1}$.
\end{enumerate}
\end{lemma}

\begin{proof}[Proof of Proposition \ref{dgorensteinequalWgorenstein}]
Assume first that $W$ is $n$-Gorenstein.
Since $A \in{}^\perp W$ and $W$ is an Ext-injective cogenerator in ${}^\perp W$, we have an injective coresolution of $0 \rightarrow A \rightarrow W^0 \rightarrow W^1 \rightarrow \cdots$ with $W^i \in \add W$. 
Applying Lemma \ref{Miyachilemma} to this exact sequence, we obtain that $A$ is $n$-Gorenstein and $\pdim I^i(A) \leq \pdim I^i(W)$ for $0 \leq i \leq n-1$.

Now assume that $A$ is $n$-Gorenstein. 
By Proposition \ref{strongcotiltforWAG}, $W=(\bigoplus_{0\le i\le n-1}P_i(DA)) \oplus \Omega^n(DA)$.
Clearly $P_i(DA)$ is $n$-Gorenstein. Moreover there exists an exact sequence $0 \rightarrow\Omega^n(DA) \rightarrow P_{n-1} \rightarrow \cdots \rightarrow P_0 \rightarrow DA \rightarrow 0$ with $P_i$ projective. Applying Lemma \ref{Miyachilemma} with $M_{-1}=\Omega^n(DA)$ to this exact sequence, we obtain that $\Omega^n(DA)$ is $n$-Gorenstein and $\pdim I^i(W) \leq \pdim I^i(A)$ for $0 \leq i \leq n-1$. Thus also $W$ is $n$-Gorenstein.
\end{proof}

We give an example of a $d$-Gorenstein algebra which has a $\Ext$-maximal cotilting module of injective dimension $d$.

\begin{example}\label{eg:nonIG with Extmax cotilt}
Let $A=KQ/I$ be a Nakayama algebra given by the following quiver with relations:
\[Q=\Bigg[
\adjustbox{valign=m}{
\begin{tikzpicture}
\node[scale=0.9] at (0,0) {\begin{tikzcd}[cells={nodes={scale=0.9}}]	1 & 2 \\
	4 & 3
	\arrow["{a_1}", from=1-1, to=1-2]
	\arrow["{a_2}", from=1-2, to=2-2]
	\arrow["{a_3}", from=2-2, to=2-1]
	\arrow["{a_4}", from=2-1, to=1-1]
\end{tikzcd}};
\end{tikzpicture}}
\,\Bigg]
,\ \ \ I=\langle a_1 a_2 a_3, a_2 a_3 a_4 , a_3 a_4 a_1 a_2 \rangle.\]
It is direct to check that $A$ is 3-Gorenstein,
and that $W:=e_1 A/e_1 J^2 \oplus e_2 A \oplus e_3 A \oplus e_4 A =\sm{1\\2}\oplus\sm{2\\3\\4}\oplus\sm{3\\4\\1\\2}\oplus\sm{4\\1\\2\\3}$ is a cotilting $A$-module of injective dimension 3, which is also Ext-maximal by Proposition \ref{CMcharapropo}(iii).
Note that $A$ is not Iwanaga-Gorenstein, and not Cohen-Macaulay since $\End_A(W)$ is not isomorphic to $A$, see Proposition \ref{dualizing = 2sided Extmax+DCP}.
\end{example}

We summarise the obtained results in this section for the special case of Cohen-Macaulay algebras in the following corollary: 

\begin{corollary}\label{syzygiesACM3}
Let $A$ be a Cohen-Macaulay algebra with dualizing module $W$ with $\fidim A=d$. Then
$A$ is $d$-Gorenstein if and only if $W$ is $d$-Gorenstein.
In this case, we have 
$$\CM A= \Omega^d(\mod A) \ \mbox{ and }\ \add W =\add\big(\Big(\bigoplus_{0\le i\le d-1}P_i(DA)\Big) \oplus \Omega^d(DA)\big).$$
Moreover, $\domdim A\ge d$ holds if and only if $\domdim W\ge d$ holds.  If $A$ is non-selfinjective, then these conditions are equivalent to $\domdim A=d$ and also to $\domdim W=d$.
\end{corollary}

\begin{proof}
The first assertion is Proposition \ref{dgorensteinequalWgorenstein}. The second one follows from Corollary \ref{syzygiesACM2} and Proposition \ref{strongcotiltforWAG}.
We prove the last one. Again by Proposition \ref{dgorensteinequalWgorenstein}, $\domdim A\ge d$ if and only if $\domdim W\ge d$. Under the assumption that $A$ is non-selfinjective, they are equivalent to $\domdim A=d$ and $\domdim W=d$ respectively. In fact, if $\domdim A>d$, then either $A$ is selfinjective or $\pdim\Omega^{-d}A=d+1$ holds, a contradiction. Similarly, if $\domdim W>d$, then $\idim W=d$ implies that $W$ is projective-injective. Since the number of non-isomorphic indecomposable direct summands of $W$ coincides with that of $A$, it follows that 
$A$ is selfinjective, a contradiction.
\end{proof}

In a forthcoming work, we will introduce \emph{minimal Auslander-Cohen-Macaulay algebras} as Cohen-Macaulay algebras with $\domdim W \geq d \geq \idim W$ for some $d \geq 2$, generalising the minimal Auslander-Gorenstein algebras of \cite{IyaSol} and establishing a generalisation of the Auslander(-Solberg) correspondence.


\section{Cohen-Macaulay algebras coming from triangulated categories}

In this section, we give a general construction of Cohen-Macaulay algebras as endomorphism algebras of objects in triangulated categories satisfying a certain set of axioms.
Recall that a \emph{Serre functor} of a $k$-linear Hom-finite triangulated category $\T$ is an autoequivalence $\nu:\T\to\T$ such that there exists a bifunctorial isomorphism $D\Hom_\T(X,Y)\simeq \Hom_\T(Y,\nu X)$ for each $X,Y\in\T$. In this case, $\tau:=\nu\circ[-1]$ is called the \emph{Auslander-Reiten translation}.

We start with recalling the result of Jorgensen \cite{J} below. The triangulated category $\T/[\F]$ should be regarded as the $d=1$ case of the $d$-Calabi-Yau reduction given in \cite{IY} for $d\ge2$.

\begin{proposition}\cite{J}
Let $\T$ be a $k$-linear Hom-finite Krull-Schmidt triangulated category with Serre functor $\nu$, and $\F$ a functorially finite subcategories of $\T$ satisfying $\F=\tau\F$. Then the ideal quotient $\T/[\F]$ has a canonical structure of a triangulated category (called \emph{reduction} of $\T$ by $\F$).
\end{proposition}

The setting of our main result in this section is stated in terms of reduction.

\begin{definition}\label{axiom T,M,F}
Our setting is the following, where each subcategory is assumed to be full and closed under direct sums and direct summands.
\begin{enumerate}[\rm(a)]
\item Let $\T$ be a $k$-linear Hom-finite Krull-Schmidt triangulated category with Serre functor $\nu$, and $\M\supset\F$ functorially finite subcategories of $\T$ satisfying $\tau\M=\M=\M[2]$ and $\tau\F=\F$.
\item Let $\langle1\rangle$ be the suspension of the reduction $\overline{\T}:=\T/[\F]$, and $\overline{(-)}:\T\to\overline{\T}$ a canonical functor. Then $\overline{\M}=\overline{\M}\langle2\rangle$ and $\overline{\nu\M}=\overline{\nu\M}\langle2\rangle$.
\item (c1) $\overline{\T}(\overline{\nu\M},\overline{\M})=0$ and
(c2) $\overline{\T}(\overline{\M},\overline{\nu\M})=0$.
\item For any proper subcategory $\F'$ of $\F$, let $\underline{\T}:=\T/[\F']$ and $\underline{(-)}:\T\to\underline{\T}$ a canonical functor. Then (d1) $\underline{\T}(\underline{\nu\M},\underline{\M})\neq0$ and
(d2) $\underline{\T}(\underline{\M},\underline{\nu\M})\neq0$.
\end{enumerate}
\end{definition}

Note that the condition (a) above implies the equalities
\[\nu\M=\M[1]=\M[-1]=\nu^{-1}\M.\]

\begin{example}\label{eg:sec3 setup}
Let $\T$ be a 1-Calabi-Yau triangulated category such that $[2]=1$. Then $\tau=1$ also holds, and hence two equalities $\tau\M=\M=\M[2]$ and $\tau\F=\F$ in (a) are automatic. 
A systematic family of 1-Calabi-Yau triangulated categories $\T$ satisfying $[2]=1$ is given by the stable category $\underline{\CM} R$ of maximal Cohen-Macaulay modules over an isolated hypersurface singularity $R$ of dimension 2; see the proof of Theorem \ref{main ePie 1}.
If, moreover, $R$ is a simple singularity of dimension $2$, then the condition (b) is also satisfied for all subcategories $\F$ of $\T$; see the proof of Theorem \ref{main ePie 1}.

For a concrete example, consider $R=k[[x,y,z]]/(x^7-yz)$ -- a simple singularity $R$ of dimension $2$ and Dynkin type $A_6$.
Then we can display the Auslander-Reiten quiver of $\T:=\underline{\CM} R$ as follows.
\[\begin{tikzcd}
	1 && 2 && 3 && 4 && 5 && 6
	\arrow["{a_1}", shift left=1, from=1-1, to=1-3]
	\arrow["{b_1}", shift left=1, from=1-3, to=1-1]
	\arrow["{a_2}", shift left=1, from=1-3, to=1-5]
	\arrow["{b_2}", shift left=1, from=1-5, to=1-3]
	\arrow["{a_3}", shift left=1, from=1-5, to=1-7]
	\arrow["{b_3}", shift left=1, from=1-7, to=1-5]
	\arrow["{a_4}", shift left=1, from=1-7, to=1-9]
	\arrow["{b_4}", shift left=1, from=1-9, to=1-7]
	\arrow["{a_5}", shift left=1, from=1-9, to=1-11]
	\arrow["{b_5}", shift left=1, from=1-11, to=1-9]
\end{tikzcd}\]

We now show that conditions (c) and (d) are satisfied for $(\M=\add\{1,2,3,6\},\F=\add\{1,3,6\})$.
First, the quotient category $\overline{\T}$ has the following AR quiver, where we encircled the nodes corresponding to objects in $\ind\overline{\M}$.
\[
\overline{\T}=\T/[\add\{1,3,6\}]:\quad 
\begin{tikzcd}[cells={nodes={circle}}]
 |[draw=white,fill=white,inner sep=1pt]|\phantom{1} &&  |[draw,inner sep=1pt]|2 && |[draw=white,fill=white,inner sep=1pt]|\phantom{3} && 4\ar[rr, shift left=1] && 5\ar[ll, shift left=1] && |[draw=white,fill=white,inner sep=1pt]|\phantom{6}
\end{tikzcd}\]
and we have $\ind\overline{\M}=\{2\}$ and $\ind\overline{\nu\M}=\{5\}$. Therefore the condition (c) is satisfied.

To check the condition (d), it suffices to consider $\F'$ given by $\add\{1,3\}$, $\add\{1,6\}$, and $\add\{3,6\}$; these have the following respective AR quivers, where the blue nodes represent objects in $\ind\underline{\F}$.
\[
\begin{tikzpicture}[
mutable/.style={draw,circle,inner sep=1pt,outer sep=3pt},
frozen/.style={draw=cyan,text=cyan,fill=white,circle,inner sep=1pt,outer sep=3pt},
frozenOut/.style={draw=white,white,fill=white,circle,inner sep=1pt,outer sep=3pt}
]
\matrix[column sep=.9cm,row sep=5pt]{
\node[] {$\T/[\add\{1,3\}]$ :}; &
\node[frozenOut] (u1) {1}; &&
\node[mutable]   (u2) {2}; &&
\node[frozenOut] (u3) {3}; &&
\node (u4) {4}; &&
\node (u5) {5}; && 
\node[frozen]    (u6) {6}; \\
\node[] {$\T/[\add\{1,6\}]$ :}; &
\node[frozenOut] (v1) {1}; &&
\node[mutable]   (v2) {2}; &&
\node[frozen]    (v3) {3}; &&
\node (v4) {4}; &&
\node (v5) {5}; &&
\node[frozenOut] (v6) {6}; \\
\node[align=left] {$\T/[\add\{3,6\}]$ :}; &
\node[frozen]    (w1) {1}; &&
\node[mutable]   (w2) {2}; &&
\node[frozenOut] (w3) {3}; &&
\node (w4) {4}; &&
\node (w5) {5}; &&
\node[frozenOut] (w6) {6}; \\
};
\begin{scope}[
shu/.style={->,transform canvas={yshift=.6ex}},
shd/.style={->,transform canvas={yshift=-.6ex}},
]
\draw[shu] (u4) -- (u5); \draw[shu] (u5) -- (u6);
\draw[shd] (u6) -- (u5); \draw[shd] (u5) -- (u4);

\draw[shu] (v2) -- (v3);\draw[shu] (v3) -- (v4);\draw[shu] (v4) -- (v5);
\draw[shd] (v5) -- (v4);\draw[shd] (v4) -- (v3);\draw[shd] (v3) -- (v2);

\draw[shu] (w1) -- (w2);\draw[shu] (w4) -- (w5);
\draw[shd] (w2) -- (w1);\draw[shd] (w5) -- (w4);
\end{scope}
\end{tikzpicture}
\]
Then the condition (d) is satisfied for $\F'=\add\{1,3\}$ since $6\in\M$, $5\in\nu\M$ and $\T/[\F'](5,6)\neq0\neq\T/[\F'](6,5)$. Also it is satisfied for $\F'=\add\{1,6\}$ since $3\in\M$, $4\in\nu\M$ and $\T/[\F'](3,4)\neq0\neq\T/[\F'](4,3)$. Finally it is satisfied for $\F'=\add\{3,6\}$ since $1\in\M\cap\nu\M$ and $\T/[\F'](1,1)\neq0$.
\end{example}

\begin{theorem}\label{main triangulated G}
Under the assumptions (a)--(d), assume further that $\M$ has an additive generator $M$. Then the following assertions hold.
\begin{enumerate}[\rm(1)]
\item $A:=\End_{\M}(M)$ is a Cohen-Macaulay algebra.
\item If $\M=\F$, then $\fidim A=0$. Otherwise, $\fidim A=2$.
\item If $\M=\nu\M$, then $A$ is selfinjective. If $\M\neq\nu\M$ and $\F=\nu\F$, then $\domdim A=2$. 
\end{enumerate}
\end{theorem}

In the rest of this section, we give a proof of Theorem \ref{main triangulated G}.
We remark that the last conditions (d) are necessary only in the last 2 steps: Lemma \ref{describe top G} and Proposition \ref{U is dualizing G}, which shows the two-sided Ext-maximality (Proposition \ref{dualizing = 2sided Extmax+DCP} (i)(ii)) of a certain cotilting module $U$.  Without (d), $U$ still satisfies $\idim U\le 2$ and $\End_A(U)\simeq A$; c.f. discussion in Example \ref{eg:non-CM}(3).

\begin{example}\label{eg:sec3 loewy}
We continue with Example \ref{eg:sec3 setup} where $\T$ comes from simple singularity of type $A_6$ and $(\M=\add\{1,2,3,6\},\F=\add\{1,3,6\})$.
Let us present $A$ by drawing the Loewy structure of their indecomposable projective and indecomposable injective modules.  
The calculation is simple using the fact that $\T\simeq \proj \Pi$ where $\Pi$ is the preprojective algebra of type $A_6$.
Recall that the Loewy structure of $\Pi = D\Pi$ is given by:
\[
\Pi:\;\; \sm{1\\2\\3\\4\\5\\6} \oplus 
\sm{\ssp2\ssp\ssp\ssp\ssp\\1\ssp3\ssp\ssp\ssp\\ \ssp2\ssp4\ssp\ssp \\ \ssp\ssp3\ssp5\ssp\\ \ssp\ssp\ssp4\ssp6\\ \ssp\ssp\ssp\ssp5\ssp}\oplus
\sm{\ssp\ssp3\ssp\ssp\ssp\\ \ssp2\ssp4\ssp\ssp\\ 1\ssp3\ssp5\ssp\\\ssp2\ssp4\ssp6\\ \ssp\ssp3\ssp5\ssp\\ \ssp\ssp\ssp4\ssp\ssp}\oplus 
\sm{\ssp\ssp\ssp4\ssp\ssp\\ \ssp\ssp3\ssp5\ssp\\ \ssp2\ssp4\ssp6\\ 1\ssp3\ssp5\ssp\\ \ssp2\ssp4\ssp\ssp\\ \ssp\ssp3\ssp\ssp\ssp}\oplus 
\sm{
\ssp\ssp\ssp\ssp5\ssp\\ \ssp\ssp\ssp4\ssp6\\ \ssp\ssp3\ssp5\ssp\\ \ssp2\ssp4\ssp\ssp\\ 1\ssp3\ssp\ssp\ssp\\ \ssp2\ssp\ssp\ssp\ssp}\oplus \sm{6\\5\\4\\3\\2\\1}
\]
Then $A = e\Pi e$ where $e=e_1+e_2+e_3+e_6$, so the Loewy structure of $P_i:=e_iA \in \proj A$ and $I_i:=D(Ae_i)\in \inj A$ are just giving by deleting the composition factors $4,5$ (now coloured in light grey) from that of $e_i\Pi$ and $D(\Pi e_i)$ respectively.

\[
A:\;\; \sm{1\\2\\3\\\gr{4}\\\gr{5}\\6} \oplus 
\sm{\ssp2\ssp\ssp\ssp\ssp\\1\ssp3\ssp\ssp\ssp\\ \ssp2\ssp\gr{4}\ssp\ssp \\ \ssp\ssp3\ssp\gr{5}\ssp\\ \ssp\ssp\ssp\gr{4}\ssp6\\ \ssp\ssp\ssp\ssp\gr{5}\ssp}\oplus
\sm{\ssp\ssp3\ssp\ssp\ssp\\ \ssp2\ssp\gr{4}\ssp\ssp\\ 1\ssp3\ssp\gr{5}\ssp\\\ssp2\ssp\gr{4}\ssp6\\ \ssp\ssp3\ssp\gr{5}\ssp\\ \ssp\ssp\ssp\gr{4}\ssp\ssp}
\oplus \sm{6\\\gr{5}\\\gr{4}\\3\\2\\1}
\qquad DA:\;\; \sm{6\\\gr{5}\\\gr{4}\\3\\2\\1} \oplus 
\sm{
\ssp\ssp\ssp\ssp\gr{5}\ssp\\ \ssp\ssp\ssp\gr{4}\ssp6\\ \ssp\ssp3\ssp\gr{5}\ssp\\ \ssp2\ssp\gr{4}\ssp\ssp\\ 1\ssp3\ssp\ssp\ssp\\ \ssp2\ssp\ssp\ssp\ssp}\oplus
\sm{\ssp\ssp\ssp\gr{4}\ssp\ssp\\ \ssp\ssp3\ssp\gr{5}\ssp\\ \ssp2\ssp\gr{4}\ssp6\\ 1\ssp3\ssp\gr{5}\ssp\\ \ssp2\ssp\gr{4}\ssp\ssp\\ \ssp\ssp3\ssp\ssp\ssp}\oplus
\sm{1\\2\\3\\\gr{4}\\\gr{5}\\6} \]
 Note that $\add\{P_X\mid X\in\F\}\supset \proj A\cap \inj A\simeq \M\cap\nu\M$ thanks to condition (c), and this example shows that the inclusion can be strict.
Also, using these diagrams one can check that $\pdim I_3\neq 0$ and $0\to A\to I_1\oplus I_3^{\oplus 2} \oplus I_6^{\oplus 3}$ is an injective hull -- this shows that $A$ is not even $1$-Gorenstein.
\end{example}

We first recall the following simple observation, which is a special case of Auslander's defect formula.

\begin{proposition}\label{defect formula}
Let $\T$ be a triangulated category with Serre functor $\nu$. For each triangle $Z\xrightarrow{g} Y\xrightarrow{f} X\xrightarrow{e} Z[1]$ in $\T$ and object $T\in\T$,
the map $(f\cdot):\T(T,Y)\to\T(T,X)$ is surjective if and only if the map $(\cdot g):\T(Y,\tau T)\to\T(Z,\tau T)$ is surjective.
\end{proposition}

\begin{proof}
We have an exact sequence $\T(T,Y)\xrightarrow{f\cdot}\T(T,X)\xrightarrow{e\cdot}\T(T,Z[1])\xrightarrow{g[1]\cdot}\T(T,Y[1])$, where the right map is isomorphic to the dual of $(\cdot g):\T(Y,\tau T)\to\T(Z,\tau T)$ by Serre duality.
Thus both conditions are equivalent to that the map $(e\cdot):\T(T,X)\to\T(T,Z[1])$ being zero.
\end{proof}

\begin{lemma}\label{2 triangles exist G}
For each $X\in\T$, there exist triangles
\begin{equation}\label{2 triangles G}
\Omega_X\xrightarrow{g} F_X\xrightarrow{f} X\to\Omega_X[1]\ \mbox{ and }\ \Omega^2_X\xrightarrow{g'}F'_X\xrightarrow{f'}\Omega_X\to X[1]
\end{equation}
with minimal right $\F$-approximations $f$, $f'$ and (not necessarily minimal) left $\F$-approximations $g$, $g'$ such that $\Omega^2_X=X\langle-2\rangle$ in $\overline{\T}$.
\end{lemma}

\begin{proof}
Take triangles
\[\Omega_X\xrightarrow{g} F_X\xrightarrow{f} X\to\Omega_X[1]\ \mbox{ and }\ \Omega^2_X\xrightarrow{g'}F'_X\xrightarrow{f'}\Omega_X\to X[1]\]
with minimal right $\F$-approximations $f$, $f'$. Then $\Omega_X=X\langle-1\rangle$ and $\Omega^2_X=X\langle-2\rangle$ holds in $\overline{\T}$. Since $\tau\F=\F$ holds by our assumption (a), Proposition \ref{defect formula} implies that $g$ and $g'$ are left $\F$-approximations.
\end{proof}

Let $\overline{A}:=\End_{\overline{\M}}(M)$ be a factor algebra of $A$. For $X\in\M$, let
\begin{align*}
P_X:=\M(M,X)\in\proj A,&\ \ \ I_X:=D\M(X,M)\simeq\T(M,\nu X)\in\inj A,\\
\overline{P}_X:=\overline{\M}(M,X)\in\proj\overline{A},&\ \ \ \overline{I}_X:=D\overline{\M}(X,M)\in\inj\overline{A}.
\end{align*}

\begin{example}\label{eg:sec3 approx seq}
We continue with Examples \ref{eg:sec3 setup} and \ref{eg:sec3 loewy} where $\T$ comes from simple singularity of type $A_6$ and $(\M=\add\{1,2,3,6\},\F=\add\{1,3,6\})$.
For the triangles in Lemma \ref{2 triangles exist G}, since $X=2\in\M$ is the unique indecomposable object not in $\F$, we only need to consider its associated triangles; otherwise, the sequences are just $0\to Y\to Y\to 0$ for any $Y\in\F$.
In this case, the two sequences are isomorphic:
\[
\xymatrix@R=3pt@C=35pt{
\ \phantom{=}(\Omega_X\ar[r]^{g} & F_X\ar[r]^{f} & X\ar[r] & \Omega_X[1]) \\
=(\phantom{ab}2 \ar[r]^{(b_1,a_2)} & 1\oplus 3\ar[r]^(.6){(a_1,-b_2)^{\mathrm{T}}} & 2\ar[r] & 5\phantom{abc}) \\
=(\Omega^2_X\ar[r]^{g'} & F'_X\ar[r]^{f'} & \Omega_X\ar[r] & \Omega^2_X[1])
}
\]
\end{example}

\begin{lemma}\label{2 exact sequences G}
The following assertions hold.
\begin{enumerate}[\rm(1)]
\item $X\mapsto\Omega^2_X$ gives a permutation of $\ind\M\setminus\ind\F$.
\item For each $X\in\M$, the triangles \eqref{2 triangles G} give exact sequences
\begin{align*}
0\to\overline{P}_{\tau X}\to\T(M,\nu\Omega_X)\xrightarrow{\nu g\cdot} I_{F_X}\xrightarrow{I_f} I_X\to0,\\
0\to\overline{I}_{\Omega^2_X}\to I_{\Omega^2_X}\xrightarrow{I_{g'}} I_{F'_X}\xrightarrow{\nu f'\cdot} \T(M,\nu\Omega_X)\to0.
\end{align*}
\end{enumerate}
\end{lemma}

\begin{proof}
(1) This is clear since the map gives the autoequivalence $\langle-2\rangle:\overline{\M}\to\overline{\M}$

(2)(i) Applying $\T(M,\nu-)$ to the left triangle in \eqref{2 triangles G}, we have an exact sequence
\[\T(M,\tau F_X)\xrightarrow{\tau f\cdot}\T(M,\tau X)\to\T(M,\nu \Omega_X)\xrightarrow{\nu g\cdot}\T(M,\nu F_X)\xrightarrow{\nu f\cdot}\T(M,\nu X).\]
The right-most map is can be written as $I_f:I_{F_X}\to I_X$. It is isomorphic to $(f\cdot):\T(\nu^{-1}M,F_X)\to\T(\nu^{-1}M,X)$ and hence surjective by $\nu^{-1}\M=\nu\M$ and our assumption (C1).
The left-most map can be written as $P_{\tau F_X}\to P_{\tau X}$, whose cokernel is $\overline{P}_{\tau X}$ since $\tau\F=\F$ and hence $\tau f$ is a right $\F$-approximation. Thus we obtain the first sequence.

(ii) Applying $\T(-,M[-1])$ to the right triangle in \eqref{2 triangles G} and using $\M[-1]=\nu\M$ and our assumption (c2), we have a surjection $(\cdot g'):\T(F'_X,M[-1])\to\T(\Omega^2_X,M[-1])$.
By Proposition \ref{defect formula}, $(f'\cdot):\T(\nu^{-1}M,F'_X)\to\T(\nu^{-1}M,\Omega_X)$ is surjective. Thus applying $\T(M,\nu-)$ to the right triangle in \eqref{2 triangles G} gives an exact sequence
\[\T(M,\nu\Omega^2_X)\xrightarrow{\nu g'\cdot}\T(M,\nu F'_X)\xrightarrow{\nu f'\cdot}\T(M,\nu\Omega_X)\to0.\]
The left map can be written as $I_{g'}:I_{\Omega^2_X}\to I_{F'_X}$, and its dual $\T(F'_X,M)\to\T(\Omega^2_X,M)$ has a cokernel $\overline{\T}(\Omega^2_X,M)$ since $g'$ is a left $\F$-approximation. Thus the kernel of the left map is $\overline{I}_{\Omega^2_X}$, and we obtain the second sequence.
\end{proof}


Combining two exact sequences above, we get a commutative diagram of exact sequences.
\begin{equation}\label{PB0 G}
\xymatrix@R=1.5em{
&0&0\\
0\ar[r]&\overline{P}_{\tau X}\ar[r]\ar[u]& \T(M,\nu\Omega_X)\ar[r]^(.6){\nu g\cdot}\ar[u]& I_{F_X}\ar[r]^{I_f}& I_X\ar[r]& 0\\
0\ar[r]&U_X\ar[r]\ar[u]&I_{F'_X}\ar[r]^{I_{gf'}}\ar[u]^{\nu f'\cdot}& I_{F_X}\ar[r]^{I_f}\ar@{=}[u]& I_X\ar[r]\ar@{=}[u]& 0\\
&I_{\Omega^2_X}\ar@{=}[r]\ar[u]&I_{\Omega^2_X}\ar[u]^{I_{g'}}\\
&\overline{I}_{\Omega^2_X}\ar[u]\ar@{=}[r]&\overline{I}_{\Omega^2_X}\ar[u]\\
&0\ar[u]&0.\ar[u]
}\end{equation}
In particular, we obtain exact sequences
\begin{align}\label{UIII sequence G}
&0\to U_X\to I_{F'_X}\xrightarrow{I_{gf'}} I_{F_X}\xrightarrow{I_f} I_X\to0,&\\ \label{IIUP sequence G}
&0\to \overline{I}_{\Omega^2_X}\to I_{\Omega^2_X}\to U_X\to\overline{P}_{\tau X}\to0.&
\end{align}

\begin{example}\label{eg:sec3 4-term seq}
In the setting of Examples \ref{eg:sec3 setup}, \ref{eg:sec3 loewy} and \ref{eg:sec3 approx seq}, the quotient algebra $\overline{A}\simeq A/A(e_1+e_3+e_6)A$ is simple and embeds in $\mod A$ as a simple module $\overline{A}_A = S_2$. Also we have
\[
\overline{P}_{\tau 2} = 2 = \overline{I}_{\Omega_2^2}, \;\;\T(M,\nu\Omega_2) = \T(M,5)=I_2=\sm{
\ssp\ssp\ssp\ssp\gr{5}\ssp\\ \ssp\ssp\ssp\gr{4}\ssp6\\ \ssp\ssp3\ssp\gr{5}\ssp\\ \ssp2\ssp\gr{4}\ssp\ssp\\ 1\ssp3\ssp\ssp\ssp\\ \ssp2\ssp\ssp\ssp\ssp} \quad\text{and}\quad U_2 =\sm{
\ssp\ssp\ssp\ssp\gr{5}\ssp\\ \ssp\ssp\ssp\gr{4}\ssp6\\ \ssp\ssp3\ssp\gr{5}\ssp\\ \ssp22\gr{4}\ssp\ssp\\ 1\ssp3\ssp\ssp\ssp\\ \ssp\gr{2}\ssp\ssp\ssp\ssp},
\]
which means that the sequence \eqref{UIII sequence G} takes the following form
\[
0\to \sm{\ssp\ssp\ssp\ssp\gr{5}\ssp\\ \ssp\ssp\ssp\gr{4}\ssp6\\ \ssp\ssp3\ssp\gr{5}\ssp\\ \ssp22\gr{4}\ssp\ssp\\ 1\ssp3\ssp\ssp\ssp\\ \ssp\gr{2}\ssp\ssp\ssp\ssp}\to 
 \sm{\ssp\ssp\ssp\gr{4}\ssp\ssp\\ \ssp\ssp3\ssp\gr{5}\ssp\\ \ssp2\ssp\gr{4}\ssp6\\ 1\ssp3\ssp\gr{5}\ssp\\ \ssp2\ssp\gr{4}\ssp\ssp\\ \ssp\ssp3\ssp\ssp\ssp} \oplus \sm{6\\ \gr{5} \\\gr{4} \\3\\2\\1} 
\to \sm{\ssp\ssp\ssp\gr{4}\ssp\ssp\\ \ssp\ssp3\ssp\gr{5}\ssp\\ \ssp2\ssp\gr{4}\ssp6\\ 1\ssp3\ssp\gr{5}\ssp\\ \ssp2\ssp\gr{4}\ssp\ssp\\ \ssp\ssp3\ssp\ssp\ssp} \oplus \sm{6\\ \gr{5} \\\gr{4} \\3\\2\\1} 
\to \sm{
\ssp\ssp\ssp\ssp\gr{5}\ssp\\ \ssp\ssp\ssp\gr{4}\ssp6\\ \ssp\ssp3\ssp\gr{5}\ssp\\ \ssp2\ssp\gr{4}\ssp\ssp\\ 1\ssp3\ssp\ssp\ssp\\ \ssp2\ssp\ssp\ssp\ssp} \to 0.
\]
\end{example}

Now let
\[I_{\F}:=\bigoplus_{X\in\ind\F}I_X,\ \ \ U_{\F^{\c}}:=\bigoplus_{X\in\ind\M\setminus\ind\F}U_X\ \mbox{ and }\ U:=I_{\F}\oplus U_{\F^{\c}}.\]
Now $U$ can be interpreted as the second simultaneous mutation of the cotilting $A$-module $DA$ with respect to the direct summand $I_{\F}$ of $DA$.


\begin{proposition}\label{U is cotilting G}
Under the assumptions (a),(b) and (c), the following assertions hold.
\begin{enumerate}[\rm(1)]
\item The $A$-module $U$ is cotilting and given by $U=\mu^{+2}_{I_{\F}}(DA)$, where $\mu^+_{I_{\F}}$ is defined in \eqref{mu^+_V}.
\item If $\F=\M$, then $\idim U=0$. Otherwise $\idim U=2$.
\end{enumerate}
\end{proposition}

\begin{proof}
It suffices to show that the sequence obtained by applying $\Hom_A(I_{\F},-)$ to \eqref{UIII sequence G} is still exact.
By our construction, the sequence $\T(\F,F'_X)\to\T(\F,F_X)\to\T(\F,X)\to0$ is exact.
Thus the sequence
\[\Hom_A(P_{\F},P_{F'_X})\to\Hom_A(P_{\F},P_{F_X})\to\Hom_A(P_{\F},P_X)\to0\]
is also exact, where $P_{\F}:=\bigoplus_{X\in\ind\F}P_X$. Since the Nakayama functor $\nu:\proj A\simeq\inj A$ is an equivalence, the sequence
\[\Hom_A(I_{\F},I_{F'_X})\to\Hom_A(I_{\F},I_{F_X})\to\Hom_A(I_{\F},I_X)\to0\]
is exact, as desired.
\end{proof}

Now we prove that $\End_A(U)$ is isomorphic to $A$. We need the following preparations.

\begin{lemma}\label{2 triangles exist2 G}
For $X\in\M$, take triangles \eqref{2 triangles G} for $\nu X$:
\begin{equation}\label{2 triangles G 2}
\Omega_{\nu X}\xrightarrow{g} F_{\nu X}\xrightarrow{f}\nu X\to \Omega_{\nu X}[1]\ \mbox{ and }\ \Omega^2_{\nu X}\xrightarrow{g'}F'_{\nu X}\xrightarrow{f'}\Omega_{\nu X}\to \Omega^2_{\nu X}[1].
\end{equation}
Then $f$ and $f'$ are minimal right $\M$-approximations and $g$ and $g'$ are (not necessarily minimal) left $\M$-approximations. 
\end{lemma}

\begin{proof}
$f$ is a minimal right $\M$-approximation by our assumption (c2). We have $\nu X\in\nu\M$ and $\Omega^2_{\nu X}\in\overline{\nu\M}\langle-2\rangle=\overline{\nu\M}$ by our assumption (b). Thus $g'$ is a left $\M$-approximation by our assumption (c1). Since $\tau\M=\M$ holds by our assumption (a), Proposition \ref{defect formula} implies that $g$ is a left $\M$-approximations, and $f'$ is a minimal right $\M$-approximation.
\end{proof}

We obtain the following vanishing properties.

\begin{lemma}\label{vanish G}
The following assertions hold.
\begin{enumerate}[\rm(1)]
\item We have $\Ext^i_A(\mod\overline{A},U)=0$ for $i=0,1$.
\item We have $\Ext^i_A(DA,\mod\overline{A})=0$ for $i=0,1$.
\item For each $X\in\M$, the triangles \eqref{2 triangles G 2} give an exact sequence $P_{F'_{\nu X}}\to P_{F_{\nu X}}\xrightarrow{f\cdot} I_X\to0$ with projective cover $(f\cdot)$.
\end{enumerate}
\end{lemma}

\begin{proof}
(1) It suffices to show that $\Ext^i_A(\mod\overline{A},U_X)=0$ for each $X\in\M$ and $i=0,1$.
We have an injective resolution \eqref{UIII sequence G} of $U_X$.
Since $\Hom_A(\mod\overline{A},I_{F'_X}\oplus I_{F_X})=0$ holds, we have the assertion.

(2)(3) It suffices to show that $\Ext^i_A(I_X,\mod\overline{A})=0$ for each $X\in\M$ and $i=0,1$.
Applying $\T(M,-)$ to the triangles in Lemma \ref{2 triangles exist2 G}, we obtain exact sequences
\[\T(M,\Omega_{\nu X})\to P_{F_{\nu X}}\to I_X\to0\ \mbox{ and }\ P_{F'_{\nu X}}\to\T(M,\Omega_{\nu X})\to0.\]
Combining them, we obtain a projective presentation of $P_{F'_{\nu X}}\to P_{F_{\nu X}}\xrightarrow{f\cdot} I_X\to0$ of $I_X$.
Since $\Hom_A(P_{F'_{\nu X}}\oplus P_{F_{\nu X}},\mod\overline{A})=0$ holds, we have the assertion.
\end{proof}

Now we are ready to prove the following result.

\begin{proposition}\label{End(U) G}
Under the assumptions (a),(b) and (c), we have $\End_A(U)\simeq A$ as algebras.
\end{proposition}

\begin{proof}
By Lemma \ref{2 exact sequences G}(1), the sequence \eqref{IIUP sequence G} gives an exact sequence
\begin{equation*}
0\to \overline{I}_{\F^{\c}}\to I_{\F^{\c}}\xrightarrow{a}U_{\F^{\c}}\to\overline{P}_{\F^{\c}}\to0,\end{equation*}
where 
$\overline{I}_{\F^{\c}}:=\bigoplus_{X\in\ind\M\setminus\ind\F}\overline{I}_X$ $I_{\F^{\c}}:=\bigoplus_{X\in\ind\M\setminus\ind\F} I_X$, $P_{\F^{\c}}:=\bigoplus_{X\in\ind\M\setminus\ind\F}P_X$ and $\overline{P}_{\F^{\c}}:=\bigoplus_{X\in\ind\M\setminus\ind\F}\overline{P}_X$.
For $b:=1_{I_{\F}}\oplus a:DA=I_{\F}\oplus I_{\F^{\c}}\to I_{\F}\oplus U_{\F^{\c}}=U$, we have an exact sequence
\begin{equation}\label{PB I_m U_m0 G}
0\to \overline{I}_{\F^{\c}}\to DA\xrightarrow{b}U\to\overline{P}_{\F^{\c}}\to0.
\end{equation}
As Lemma \ref{vanish G}(1) asserts that  $\Ext_A^i(\overline{I}_{\F^{\c}}\oplus\overline{P}_{\F^{\c}},U)=0$ holds for $i=0,1$, by applying $\Hom_A(-,U)$ to \eqref{PB I_m U_m0 G}, we have an isomorphism
\[(\cdot b):\End_A(U)\simeq\Hom_A(DA,U).\]
Since $\Ext^i_A(DA,\overline{I}_{\F^{\c}}\oplus\overline{P}_{\F^{\c}})=0$ holds for $i=0,1$ by Lemma \ref{vanish G}(2), by applying $\Hom_A(DA,-)$ to \eqref{PB I_m U_m0 G}, we have an isomorphism
\[(b\cdot):\End_A(DA)\simeq\Hom_A(DA,U).\]
Thus we get the desired algebra isomorphism
\[\End_A(U)\xrightarrow{\cdot b}\Hom_A(DA,U)\xrightarrow{(b\cdot)^{-1}}\End_A(DA)=A.\qedhere\]
\end{proof}

The assumptions (d) are necessary only in the results below.
For each $X\in\ind\M$, let
\[S_X:=\top P_X\]
be the corresponding simple $A$-module. Let
\[S_{\F}:=\bigoplus_{X\in\ind\F}S_X\ \mbox{ and }\ S_{\F^{\c}}:=\bigoplus_{X\in\ind\M\setminus\ind\F}S_X.\]
We need the following easy but technical observations.

\begin{lemma}\label{describe top G}
The following assertions hold.
\begin{enumerate}[\rm(1)]
\item We have $\top DA\in\add S_{\F}$.
\item For each $X\in\ind\M\setminus\ind\F$, we have $S_{\tau X}\in\add\top U_X\subset\add(S_{\tau X}\oplus S_{\F})$.
\item We have $S_{\F^{\c}}\in\add\top U_{\M\setminus \F}$.
\item If (d2) holds, then we have $\add S_{\F}=\add\top I_{\F}$.
\item If (d1) holds, then for each $X\in\ind\F$, we have $I_X\notin\Fac(I_{\F}/I_X)$.
\end{enumerate}
\end{lemma}

\begin{proof}
(1) By Lemma \ref{vanish G}(2), we have $\Hom_A(DA,\mod\overline{A})=0$. Thus the assertion follows.

(2) We have an exact sequence $I_{\Omega^2_X}\to U_X\to\overline{P}_{\tau X}\to0$ in \eqref{IIUP sequence G} . Since $\top I_{\Omega^2_X}\in\add S_{\F}$ by (1) and $\top\overline{P}_{\tau X}=S_{\tau X}$, we have the assertion.

(3) Since $\tau$ gives a permutation of $\ind\M\setminus\ind\F$, the assertion follows from (2).

(4) Let $F:=\bigoplus_{X\in\ind\F}X$. By Lemma \ref{vanish G}(3), we obtain a projective cover $P_{F_{\nu F}}\to I_{\F}\to0$. Thus $\top I_{\F}\simeq S_{F_{\nu F}}$ holds.
On the other hand, by our assumption (c1), any morphism from $\M$ to $\nu\M$ factors through $\nu\F=\add\nu F$ and hence also factors through $\add F_{\nu F}$. By our assumption (d2), we have $\add F_{\nu F}=\F$. Thus $\add\top I_{\F}=\add S_{F_{\nu F}}=\add S_{\F}$ holds.

(5) Assume $I_X\in\Fac(I_{\F}/I_X)$, and let $\F':=\add(\ind\F\setminus\{X\})$. Take a right $\F'$-approximation $f:G\to X$. Applying $\T(M,\nu-)$, we obtain a morphism $I_G\to I_X$ which is a right $\add(I_{\F}/I_X)$-approximation and hence surjective by our assumption $I_X\in\Fac(I_{\F}/I_X)$.
Thus any morphism from $\M$ to $\nu X$ factors through $\nu\F'$. By (c1), any morphism from $\M$ to $\nu\M$ factors though $\nu\F'$, a contradiction to (d1). Thus $I_X\notin\Fac(I_{\F}/I_X)$ holds.
\end{proof}

We are ready to prove the following result.

\begin{proposition}\label{U is dualizing G}
Under the assumptions (a)--(d), $A$ is a Cohen-Macaulay algebra with dualizing module $U$.
\end{proposition}

\begin{proof}
By Propositions \ref{U is cotilting G} and \ref{End(U) G}, $U$ is a cotilting $A$-module with $\End_A(U)\simeq A$. It remains to show that $U$ is $\Ext$-maximal on both sides.

(i) By Lemma \ref{describe top G}(3)(4), $\add\top U$ contains all simple $A$-modules. By Proposition \ref{CMcharapropo}(ii)$\Rightarrow$(i), $U$ is an $\Ext$-maximal cotilting $A^{\op}$-module.

(ii) We now show that $U$ is an $\Ext$-maximal cotilting $A$-module.
By Proposition \ref{CMcharapropo}(iii)$\Rightarrow$(i), it suffices to show that, for each indecomposable direct summand $V$ of $U$ as an $A$-module, we have $V\notin\Fac(U/V)$.

For each $X\in\ind\M\setminus\ind\F$, we have $S_{\tau X}\in\top U_X$ and $S_{\tau X}\notin\top(U/U_X)$ by Lemma \ref{describe top G}(1)(2). Thus $U_X\notin\Fac(U/U_X)$ holds.

It remains to prove that, for each $X\in\ind\F$, $I_X\notin\Fac(U/I_X)$ holds.
In fact, the sequence \eqref{UIII sequence G} shows that we have an exact sequence $0\to U_{\F^{\c}}\to I$ with $I\in\add I_{\F}$. Thus, for each $X\in\ind\F$, any morphism $U_{\F^{\c}}\to I_X$ factors through $I\in\add I_{\F}$. Therefore, if $I_X\in\Fac(U/I_X)$ holds, then $I_X\in\Fac(I_{\F}/I_X)$ holds, a contradiction to Lemma \ref{describe top G}(5).
\end{proof}

\begin{proof}[Proof of Theorem \ref{main triangulated G}]
(1) The assertion follows from Proposition \ref{U is dualizing G}.

(2) Since $A$ is Cohen-Macaulay with dualizing module $U$, we have $\idim U=\fidim A$.  Now the claim is immediate from Proposition \ref{U is cotilting G}(2).

(3) If $\M=\nu\M$, then $A$ is selfinjective by Serre duality.
Assume $\M\neq\nu\M$ and $\F=\nu\F$. Then $I_{\F}$ is projective-injective. Thus $\domdim W\ge 2$ holds by the exact sequence \eqref{UIII sequence G}. The equality holds since $\M\neq\F$. Thus $\domdim A=2$ holds by Proposition \ref{syzygiesACM3}.
\end{proof}

\section{Contracted preprojective algebras of Dynkin type}\label{section 4}

\subsection{Main results}
In this subsection, we show that contracted preprojective algebras of Dynkin type are Cohen-Macaulay. Let us consider Dynkin diagrams:
\[\xymatrix@C0.4cm@R0.2cm{
A_n&1\ar@{-}[r]&2\ar@{-}[r]&3\ar@{-}[r]&\ar@{..}[rrr]&&&\ar@{-}[r]&n-2\ar@{-}[r]&n-1\ar@{-}[r]&n\\
&&2&&&&&&\\
D_n&1\ar@{-}[r]&3\ar@{-}[u]\ar@{-}[r]&4\ar@{-}[r]&\ar@{..}[rrr]&&&\ar@{-}[r]&n-1\ar@{-}[r]&n\\
&&&6\\
E_6&1\ar@{-}[r]&2\ar@{-}[r]&3\ar@{-}[r]\ar@{-}[u]&4\ar@{-}[r]&5\\
&&&&7\\
E_7&1\ar@{-}[r]&2\ar@{-}[r]&3\ar@{-}[r]&4\ar@{-}[r]\ar@{-}[u]&5\ar@{-}[r]&6\\
&&&&&8\\
E_8&1\ar@{-}[r]&2\ar@{-}[r]&3\ar@{-}[r]&4\ar@{-}[r]&5\ar@{-}[r]\ar@{-}[u]&6\ar@{-}[r]&7}\]
We define the canonical involution $\iota$ of each Dynkin diagram as follows:
\begin{itemize}
\item For $A_n$, we put $\iota(i)=n+1-i$.
\item For $D_n$ with odd $n$, we put $\iota(1)=2$, $\iota(2)=1$ and $\iota(i)=i$ for other $i$.
\item For $E_6$, we put $\iota(1)=5$, $\iota(2)=4$, $\iota(4)=2$, $\iota(5)=1$ and $\iota(i)=i$ for other $i$.
\item For other types, we put $\iota=1$.
\end{itemize}

\begin{definition}\label{define frozen}
Let $\Delta$ be a Dynkin diagram, and $\iota$ the canonical involution of $\Delta$. We fix an arbitrary subset $J$ of the set $\Delta_0$ of vertices. 
\begin{enumerate}[\rm(1)]
\item We call $i\in J$ \emph{frozen} if there exists $i'\in J$ and a sequence $i=i_0,\ldots,i_\ell=\iota(i')$ of vertices in $\Delta$ with $\ell\ge0$ such that $i_j$ and $i_{j+1}$ are connected by an edge for each $0\le j\le\ell-1$, and each $i_j$ with $1\le j\le\ell$ does not belong to $J$.
\item We call $i\in J$ \emph{mutable} if it is not frozen.
\item Let $J_\m$ the subset of $J$ of mutable elements, and $J_\f:=J\setminus J_\m$.
\end{enumerate}
\end{definition}

Let $\Pi$ be the preprojective algebra of Dynkin type $\Delta$. It is well-known that $\Pi$ is a selfinjective algebra with Nakayama permutation $\iota$, that is, $D(\Pi e_i)\simeq e_{\iota(i)}\Pi$ as $\Pi$-modules for each $i\in\Delta_0$, see for example \cite{G} for a modern proof.

\begin{theorem}\label{main ePie 1}
Let $\Pi$ be the preprojective algebra of Dynkin type $\Delta$. For a subset $J$ of $\Delta_0$, let $e:=\sum_{i\in J}e_i\in\Pi$ and $A:=e\Pi e$.
\begin{enumerate}[\rm(1)]
\item $A$ is a Cohen-Macaulay algebra with dualizing module $\mu_{e_\f(DA)}^{+2}(DA)$, where $e_\f:=\sum_{i\in J_\f}e_i$ and $\mu^+_{e_\f(DA)}$ is defined in \eqref{mu^+_V}.
\item If $J=J_\f$, then $\fidim A=0$. Otherwise, $\fidim A=2$.
\item If $J=\iota(J)$, then $A$ is selfinjective. If $J\neq\iota(J)$ and $J_\f=\iota(J_\f)$, then $\domdim A=2$.
\end{enumerate}
\end{theorem}

We need the following easy combinatorial observation, which says that any path between $i\in J$ and $\iota(i')\in\iota(J)$ must pass through a frozen node.

\begin{lemma}\label{enough frozen}
Let $\Delta$ be a Dynkin diagram, $J$ a subset of $\Delta_0$, and $i,i'\in J$.
Let $i=i_0,\ldots,i_\ell=\iota(i')$ be a sequence of vertices in $\Delta$ with $\ell\ge0$ such that $i_j$ and $i_{j+1}$ are connected by an edge for each $0\le j\le\ell-1$.
Then some of $i_0,\ldots,i_{\ell}$ belongs to $J_\f$, and also some of $i_0,\ldots,i_{\ell}$ belongs to $\iota(J_\f)$.
\end{lemma}

\begin{proof}
Take maximal $0\le j\le\ell$ such that $i_j$ belongs to $J$. Since $i_j$ is the unique vertex in the sequence $i_j,i_{j+1},\ldots,i_\ell$ which belongs to $J$, it is frozen.

Similarly, take minimal $0\le j\le\ell$ such that $i_j$ belongs to $\iota(J)$. Looking at the sequence $\iota(i_j),\iota(i_{j-1}),\ldots,\iota(i_0)$, it follows that $\iota(i_j)$ is frozen.
\end{proof}

\begin{proof}[Proof of Theorem \ref{main ePie 1}]
Let $R$ be a simple singularity of type $\Delta$ in Krull dimension two, and let $\T:=\underline{\CM}R$ be the stable category of Cohen-Macaulay $R$-modules. Then it is well-known that $\T$ is equivalent to $\proj\Pi$ as an additive category, see for example \cite[Proposition 5.8]{AIR} and references therein. For each $i\in\Delta_0$, we denote by $X_i\in\ind\T$ the corresponding object. 
Then the canonical involution $\iota$ of $\Delta$ describes the suspension functor $[1]$ of $\T$. Let
\[\T\supset\M:=\add\{X_i\mid i\in J\}\supset\F:=\add\{X_i\mid i\in J_\f\}.\]
It suffices to show that the conditions (a)--(d) in Theorem \ref{main triangulated G} are satisfied.

(a) This is roughly explained in Example \ref{eg:sec3 setup}; we give some more details here. Since $R$ is a 2-dimensional Gorenstein isolated singularity, Auslander-Reiten duality implies that $\T$ is 1-Calabi-Yau \cite{Aus1,Y,LW}, that is, $\tau=1$.  Since $R$ is hypersurface, it is basic in the theory of matrix factorizations that $\T$ satisfies $[2]=1$ \cite{E,Y,LW}.  All subcategories of $\T$ are functorially finite.

(b) Let $e_f:=\sum_{i\in J_\f}e_i$. Then $\overline{\Pi}:=\Pi/(e_f)$ is a preprojective algebra of $\Delta\setminus J_\f$, which is a disjoint union of Dynkin diagrams. Moreover $\overline{\T}$ is equivalent to $\proj\overline{\Pi}$. By the same reason as in (a), the suspension functor of $\overline{\T}$ satisfies $\langle2\rangle=1$.

(c) We only need to consider a morphism $f:X_i\to X_{i'}[1]:$ with $i,i'\in J$ given by a path $i=i_0\to i_1\to\cdots\to i_\ell=\iota(i')$ in the double of $\Delta$. By Lemma \ref{enough frozen}, some of $i_0,\ldots,i_\ell$ belongs to $J_\f$, and some of $i_0,\ldots,i_\ell$ belongs to $\iota(J_\f)$.
Thus $f$ factors through $\F$ and also factors through $\F[1]$.

(d) We only prove (d2) since the proof of (d1) is similar.

Assume that there exists $i\in J_\f$ such that each morphism $X\to Y[1]$ with $X,Y\in\M$ factors through $\F':=\add(\ind\F\setminus\{X_i\})$. Take a sequence $i=i_0,\ldots,i_\ell=\iota(i')$ as in Definition \ref{define frozen}(1).

If $i=\iota(i')$, then there exists an isomorphism $X_i\simeq X_{i'}[1]$, which does not factor through $\F'$, a contradiction.
Assume $i\neq\iota(i')$. Without loss of generality, we can assume that the vertices $i_0,\ldots,i_\ell$ are pairwise distinct. Since each $i_j$ with $0\le j\le\ell$ does not belong to $J\setminus\{i\}$, the morphism $X_i\to X_{i'}[1]$ corresponding to the path $i_0\to i_1\to\cdots\to i_\ell$ does not factor through $\F'$, a contradiction.
\end{proof}

As an immediate consequence, we obtain the following result.

\begin{theorem} \label{ARdominantdimension2criterion}
Let $R=k[[x_0,x_1,...,x_d]]/(f_\Delta^{d})$ be a simple singularity with an algebraically closed field $k$ of characteristic 0. For each maximal Cohen-Macaulay $R$-module $M$, the stable endomorphism ring of $\underline{\End}_R(M)$ is a Cohen-Macaulay algebra in the sense of Definition \ref{define CM}.
\end{theorem}

\begin{proof}
Let $d:=\dim R$ and $E:=\underline{\End}_R(M)$. We divide into 2 cases.

(i) Assume that $d$ is even. Then the resulting stable endomorphism ring is of the form $e\Pi e$ for  a preprojective algebra $\Pi$ of Dynkin type. This is well known for Krull dimension two (see for example \cite[Proposition 5.8]{AIR}) and is a consequence of Kn\"orrer periodicity for general even Krull dimension. Now the assertion follows from Theorem \ref{main ePie 1}.

(ii) Assume that $d$ is odd. Then the Serre functor of the stable category $\underline{\CM} R$ is given by the identity. In fact, since $R$ has an isolated singularity, $[d-1]$ gives a Serre functor (that is, $\underline{\CM} R$ is $(d-1)$-Calabi-Yau) by Auslander-Reiten duality. Since $R$ is a hypersurface, $[2]=1$ by matrix factorization \cite{E} (see also \cite{Y}). Thus $[d-1]=1$ gives a Serre functor. In particular, $E\simeq DE$ holds as $E$-bimodules, and hence $E$ is a symmetric algebra. Thus it is Cohen-Macaulay.

Alternatively, one can show that $E$ is a symmetric algebra by using explicit presentations written in \cite[Theorem 5.12]{Sko}
\end{proof}

\subsection{Homological dimensions}

In this subsection, we calculate homological dimensions of contracted preprojective algebras of Dynkin type. We then apply our results to answer the Question \ref{AR conjecture} of Auslander-Reiten negatively.

Let $\Pi=\Pi(\Delta)$ be a preprojective algebra of Dynkin type. 
For a non-empty subset $J\subset\Delta_0$, let $e=e_J:=\sum_{i\in J}e_i$ and $A=\Pi(\Delta,J):=e\Pi e$ the contracted preprojective algebra.

\begin{definition}
We call a non-empty subset $J$ of $\Delta_0$ \emph{impartial} if it does \emph{not} satisfy any of the following.
\begin{enumerate}
\item $\Delta=A_n$, and $J\subset[1,\frac{n+1}{2})$ or $J\subset(\frac{n+1}{2},n]$.
\item $\Delta=D_n$ with odd $n$,and $J=\{1\}$ or $\{2\}$.
\item $\Delta=E_6$, and $J=\{1\}$ or $\{5\}$.
\end{enumerate}
\end{definition}

To study contracted preprojective algebras, we can always assume $J$ is impartial. In fact, in case (1) above, $J\subset[1,\frac{n+1}{2})$, then for $m:=2\max J-1$, $J$ is an impartial subset of $[1,m]$ and we have $\Pi(A_n,J)=\Pi(A_m,J)$. In case (2) and (3), we know $A$ explicitly: It is $k[x]/(x^{\frac{n-1}{2}})$ for (2) and $k[x]/(x^2)$ for (3), and therefore $A$ is selfinjective and non-semisimple.

Our main result below gives explicit values of three homological dimensions of contracted preprojective algebras, that is, dominant dimension, selfinjective dimension and global dimension.
Note that $\domdim A=\domdim A^{\op}$ and $\gldim A=\gldim A^{\op}$ hold in general.  Also $\idim A_A=\idim{}_AA$ holds for Cohen-Macaulay algebras by \eqref{4 dimension} and \cite[Proposition 6.10]{Ar1}, and so we can simply denote it by $\idim A$.
Notice that the assumption that $J$ is impartial simplifies the statement.

\begin{theorem}\label{dimension for type A}
Let $\Pi$ be the preprojective algebra of Dynkin type $\Delta$, $J\subset\Delta_0$ an impartial subset, and $A:=\Pi(\Delta,J)$. Then the following assertions hold.
\begin{enumerate}[\rm(1)]
\item Assume that $\Delta$ is $D_{2n}$ with $n\ge2$ or $E_n$ with $n=7,8$. Then $A$ is always selfinjective and non-semisimple. 
\item Assume $\Delta=A_n$. Then the following assertions hold.
\begin{enumerate}[\rm(i)]
\item If $J=\iota(J)$, then $\domdim A=\infty$. If $J\neq\iota(J)$ and $J_\f=\iota(J_\f)$, then $\domdim A=2$. In all remaining cases, $\domdim A=0$.
\item If $J=\iota(J)$, then $\idim A=0$. If $J\neq\iota(J)$ and $J$ is contained in $[1,\frac{n+1}{2}]$ or $[\frac{n+1}{2},n]$, then $\idim A=2$. In all remaining cases, $\idim A=\infty$.
\item If $n=1$, then $\gldim A=0$. If $n>1$ and $J$ is $[1,\frac{n+1}{2}]$ or $[\frac{n+1}{2},n]$, then $\gldim A=2$. In all remaining cases, $\gldim A=\infty$.
\end{enumerate}
\item Assume $\Delta=D_{2n+1}$ with $n\ge2$.
\begin{enumerate}[\rm(i)]
\item If $\#(\{1,2\}\cap J)$ is 0 or 2, then $\domdim A=\infty$. If $3\in J$ and $\#(\{1,2\}\cap J)=1$, then $\domdim A=2$. In all remaining cases, $\domdim A=0$.
\item If $\#(\{1,2\}\cap J)$ is 0 or 2, then $A$ is selfinjective. Otherwise, $\idim A=\infty$.
\item $\gldim A=\infty$ always holds.
\end{enumerate}
\item Assume $\Delta=E_6$. 
\begin{enumerate}[\rm(i)]
\item If $J=\iota(J)$, then $\domdim A=\infty$. If $J\neq\iota(J)$ and $J_\f=\iota(J_\f)$, then $\domdim A=2$. In all remaining cases, $\domdim A=0$.
\item If $J=\iota(J)$, then $A$ is selfinjective. Otherwise, $\idim A=\infty$.
\item $\gldim A=\infty$ always holds.
\end{enumerate}
\end{enumerate}
\end{theorem}



The proof of Theorem \ref{dimension for type A} is given in the next subsection, where we do a case study.
Under the assumption that $J$ is impartial, Theorem \ref{dimension for type A} can be summarised as follows, where the conditions on $\Delta$ and $J$ on the right-hand side are necessary and sufficient for the corresponding contracted preprojective algebra $A=\Pi(\Delta,J)$ to satisfy the homological conditions on the left-hand side.
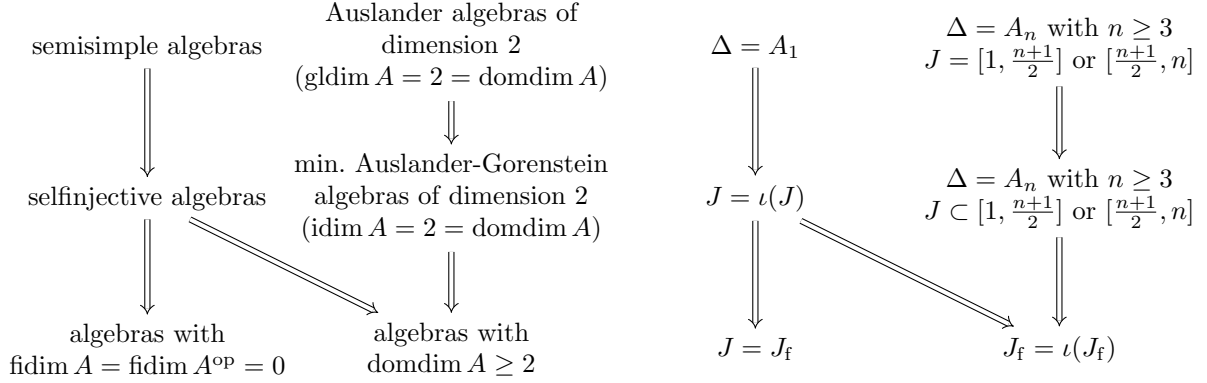
\[\begin{tikzpicture}[every text node part/.style={align=center}] 
\tikzset{
    iff/.style={implies-implies,double equal sign distance},
    imp/.style={-implies,double equal sign distance},
  }
\node (v1) at (0,0) {semisimple algebras};
\node (v2) at (0,-2) {selfinjective algebras};
\node (v3) at (0,-4) {algebras with\\ $\fidim A=\fidim A^{\op}=0$};
\node (v4) at (4,0) {Auslander algebras of\\ dimension 2\\ $(\gldim A=2=\domdim A)$};
\node (v5) at (4,-2) {min.\ Auslander-Gorenstein\\ algebras of dimension 2\\ $(\idim A=2=\domdim A)$};
\node (v6) at (4,-4) {algebras with\\ $\domdim A\ge2$};

\draw [imp] (v1) -- (v2); 
\draw [imp] (v2) -- (v3);
\draw [imp] (v4) -- (v5);
\draw [imp] (v5) -- (v6);
\draw [imp] (v2) -- (v6);

\node (v11) at (8,0) {$\Delta=A_1$};
\node (v12) at (8,-2) {$J=\iota(J)$};
\node (v13) at (8,-4) {$J=J_\f$};
\node (v14) at (12,0) {$\Delta=A_n$ with $n\ge 3$\\ $J=[1,\frac{n+1}{2}]$ or $[\frac{n+1}{2},n]$};
\node (v15) at (12,-2) {$\Delta=A_n$ with $n\ge3$\\ $J\subset[1,\frac{n+1}{2}]$ or $[\frac{n+1}{2},n]$};
\node (v16) at (12,-4) {$J_\f=\iota(J_\f)$};

\draw [imp] (v11) -- (v12); 
\draw [imp] (v12) -- (v13);
\draw [imp] (v14) -- (v15);
\draw [imp] (v15) -- (v16);
\draw [imp] (v12) -- (v16);
\end{tikzpicture}
\]


\begin{example} \label{tableexample}
We list up (not necessarily impartial) non-empty subsets $J=J_\f\sqcup J_\m$ up to $\iota$-symmetry.  Elements of $J_\f$ and $J_\m$ are coloured in blue and black respectively,
and the elements not in $J$ are coloured white. For type $D_{2n+1}$, the small dots mean that any choices are allowed.
\[\begin{array}{|c||c|c|c|c|c|}\hline
\begin{array}{c}\mbox{($\idim A,$}\\ \mbox{$\fidim A,$}\\ \mbox{$\domdim A)$}\end{array}&\mbox{$(0,0,\infty)$}&
\mbox{$(2,2,2)$}&\mbox{$(\infty,2,2)$}&
\mbox{$(\infty,2,0)$}&
\mbox{$(\infty,0,0)$}
\\ \hline\hline
A_3&\def\arraystretch{.6}\begin{array}{l}\ff\ff\ff\ \ \ff\ee\ff\\ \ee\ff\ee\ \ \ff\ee\ee\end{array}&\def\arraystretch{.6}\begin{array}{l}\mm\ff\ee\end{array}&&&\\ \hline
A_4&\def\arraystretch{.6}\begin{array}{l}\ff\ff\ff\ff\ \ \ff\ee\ee\ff\\ \ee\ff\ff\ee\ \ \ff\ee\ee\ee\\ \ee\ff\ee\ee\end{array}&
\def\arraystretch{.6}\begin{array}{l}\mm\ff\ee\ee\end{array} & \mm\ff\ff\ee &&\def\arraystretch{.6}\begin{array}{l}\ff\ff\ee\ff\\ \ff\ee\ff\ee\end{array} \\ \hline
A_5&\def\arraystretch{.6}\begin{array}{l} \ff\ff\ff\ff\ff\ \ \ff\ff\ee\ff\ff\\ \ff\ee\ff\ee\ff\ \ \ff\ee\ee\ee\ff\\ \ee\ff\ff\ff\ee\ \ \ee\ff\ee\ff\ee\\ \ff\ee\ee\ee\ee\ \ \ee\ff\ee\ee\ee\\ \ee\ee\ff\ee\ee\end{array}&
\def\arraystretch{.6}\begin{array}{l} \mm\mm\ff\ee\ee\\ \mm\ff\ee\ee\ee\\ \mm\ee\ff\ee\ee\\ \ee\mm\ff\ee\ee\end{array}&
\def\arraystretch{.6}\begin{array}{l}\mm\ff\ff\ff\ee\\ \ff\mm\ff\ee\ff\\ \mm\ff\ee\ff\ee\end{array}&&
\def\arraystretch{.6}\begin{array}{l}\ff\ff\ee\ee\ff\\ \ff\ee\ff\ff\ee\\ \ff\ee\ee\ff\ee\end{array}\\ \hline
A_6&\def\arraystretch{.6}\begin{array}{l}\ff\ff\ff\ff\ff\ff\ \ \ff\ff\ee\ee\ff\ff\\ \ff\ee\ff\ff\ee\ff\ \ \ff\ee\ee\ee\ee\ff\\ \ee\ff\ff\ff\ff\ee\ \ \ee\ff\ee\ee\ff\ee\\ \ee\ee\ff\ff\ee\ee\ \ \ff\ee\ee\ee\ee\ee\\ \ee\ff\ee\ee\ee\ee\ \ \ee\ee\ff\ee\ee\ee\end{array}&
\def\arraystretch{.6}\begin{array}{l} \mm\mm\ff\ee\ee\ee\\ \mm\ff\ee\ee\ee\ee\\ \mm\ee\ff\ee\ee\ee\\ \ee\mm\ff\ee\ee\ee\end{array}&
\def\arraystretch{.6}\begin{array}{l}\mm\ff\ff\ff\ff\ee\ \ \ff\mm\ff\ff\ee\ff\\  \mm\mm\ff\ff\ee\ee\ \ \mm\ff\ee\ee\ff\ee\\ \mm\ee\ff\ff\ee\ee\ \ \ee\mm\ff\ff\ee\ee\end{array}&
\def\arraystretch{.6}\begin{array}{l} \mm\ff\ff\ee\ff\ee\\  \ff\mm\ff\ee\ee\ff\\ \mm\ff\ee\ff\ff\ee\\ \mm\ff\ee\ff\ee\ee\end{array}&
\def\arraystretch{.6}\begin{array}{l}\ff\ff\ff\ee\ff\ff\ \ \ff\ff\ee\ff\ee\ff\\ \ff\ff\ee\ee\ee\ff\ \ \ff\ee\ff\ff\ff\ee\\ \ff\ee\ff\ee\ff\ee\ \ \ff\ee\ff\ee\ee\ff\\ \ff\ee\ee\ff\ff\ee\ \ \ff\ee\ee\ff\ee\ee\\ \ff\ee\ee\ee\ff\ee\ \ \ee\ff\ff\ee\ff\ee\\ \ee\ff\ee\ff\ee\ee\end{array}\\ \hline
D_{2n}&\mbox{Any choice}&&&&\\ \hline
D_{2n+1}&\begin{array}{l}{\def\arraystretch{0.1}\begin{array}{l}\ \,\ff\\
\ff\cdot\cdots\cdot\end{array}}
{\def\arraystretch{0.1}\begin{array}{l}\ \,\ee\\
\ee\cdot\cdots\cdot\end{array}}\\ 
{\def\arraystretch{0.1}\begin{array}{l}\ \,\ee\\ \ff\ee\ee\ee\ee\ee\end{array}}\end{array}
&&\def\arraystretch{0.1}\begin{array}{l}\ \,\ee\\ \mm\ff\cdots\cdot\end{array}&&\def\arraystretch{0.1}\begin{array}{l}\ \,\ee\\ \ff\ee\cdots\cdot\end{array}\\ \hline
E_6&\begin{array}{l} 
{\def\arraystretch{0.1}\begin{array}{c}\ff\\
\ff\ff\ff\ff\ff\end{array}}
{\def\arraystretch{0.1}\begin{array}{c}\ee\\
\ff\ff\ff\ff\ff\end{array}}\\
{\def\arraystretch{0.1}\begin{array}{c}\ff\\
\ff\ff\ee\ff\ff\end{array}}
{\def\arraystretch{0.1}\begin{array}{c}\ee\\
\ff\ff\ee\ff\ff\end{array}}\\
{\def\arraystretch{0.1}\begin{array}{c}\ff\\
\ff\ee\ff\ee\ff\end{array}}
{\def\arraystretch{0.1}\begin{array}{c}\ee\\
\ff\ee\ff\ee\ff\end{array}}\\
{\def\arraystretch{0.1}\begin{array}{c}\ff\\
\ff\ee\ee\ee\ff\end{array}}
{\def\arraystretch{0.1}\begin{array}{c}\ee\\
\ff\ee\ee\ee\ff\end{array}}\\
{\def\arraystretch{0.1}\begin{array}{c}\ff\\
\ee\ff\ff\ff\ee\end{array}}
{\def\arraystretch{0.1}\begin{array}{c}\ee\\
\ee\ff\ff\ff\ee\end{array}}\\
{\def\arraystretch{0.1}\begin{array}{c}\ff\\
\ee\ff\ee\ff\ee\end{array}}
{\def\arraystretch{0.1}\begin{array}{c}\ee\\
\ee\ff\ee\ff\ee\end{array}}\\
{\def\arraystretch{0.1}\begin{array}{c}\ff\\
\ee\ee\ff\ee\ee\end{array}}
{\def\arraystretch{0.1}\begin{array}{c}\ee\\
\ee\ee\ff\ee\ee\end{array}}\\
{\def\arraystretch{0.1}\begin{array}{c}\ff\\
\ee\ee\ee\ee\ee\end{array}}
{\def\arraystretch{0.1}\begin{array}{c}\ee\\
\ff\ee\ee\ee\ee\end{array}}
\end{array}&
\begin{array}{l} 
\end{array}&

\begin{array}{l} 
{\def\arraystretch{0.1}\begin{array}{c}\ff\\
\mm\ff\ff\ff\ee\end{array}}
{\def\arraystretch{0.1}\begin{array}{c}\ff\\
\ff\mm\ff\ee\ff\end{array}}\\
{\def\arraystretch{0.1}\begin{array}{c}\ee\\
\mm\ff\ff\ff\ee\end{array}}
{\def\arraystretch{0.1}\begin{array}{c}\ee\\
\ff\mm\ff\ee\ff\end{array}}\\
{\def\arraystretch{0.1}\begin{array}{c}\ff\\
\mm\mm\ff\ee\ee\end{array}}
{\def\arraystretch{0.1}\begin{array}{c}\ee\\
\mm\mm\ff\ee\ee\end{array}}\\
{\def\arraystretch{0.1}\begin{array}{c}\ff\\
\mm\ff\ee\ff\ee\end{array}}
{\def\arraystretch{0.1}\begin{array}{c}\ee\\
\mm\ff\ee\ff\ee\end{array}}\\
{\def\arraystretch{0.1}\begin{array}{c}\ff\\
\mm\ee\ff\ee\ee\end{array}}
{\def\arraystretch{0.1}\begin{array}{c}\ee\\
\mm\ee\ff\ee\ee\end{array}}\\
{\def\arraystretch{0.1}\begin{array}{c}\ff\\
\ee\mm\ff\ee\ee\end{array}}
{\def\arraystretch{0.1}\begin{array}{c}\ee\\
\ee\mm\ff\ee\ee\end{array}}
\end{array}&

\begin{array}{l}
{\def\arraystretch{0.1}\begin{array}{c}\ff\\
\mm\ff\ee\ee\ee\end{array}}\\
{\def\arraystretch{0.1}\begin{array}{c}\ee\\
\mm\ff\ee\ee\ee\end{array}}
\end{array}&

\begin{array}{l}
{\def\arraystretch{0.1}\begin{array}{c}\ff\\
\ff\ff\ee\ee\ff\end{array}}
{\def\arraystretch{0.1}\begin{array}{c}\ee\\
\ff\ff\ee\ee\ff\end{array}}\\
{\def\arraystretch{0.1}\begin{array}{c}\ff\\
\ff\ee\ff\ff\ee\end{array}}
{\def\arraystretch{0.1}\begin{array}{c}\ee\\
\ff\ee\ff\ff\ee\end{array}}\\
{\def\arraystretch{0.1}\begin{array}{c}\ff\\
\ff\ee\ee\ff\ee\end{array}}
{\def\arraystretch{0.1}\begin{array}{c}\ee\\
\ff\ee\ee\ff\ee\end{array}}\\
{\def\arraystretch{0.1}\begin{array}{c}\ff\\
\ff\ee\ee\ee\ee\end{array}}
{\def\arraystretch{0.1}\begin{array}{c}\ff\\
\ee\ff\ee\ee\ee\end{array}}\\
{\def\arraystretch{0.1}\begin{array}{c}\ee\\
\ee\ff\ee\ee\ee\end{array}}
\end{array}\\ \hline

E_7&\mbox{Any choice}&&&&\\ \hline
E_8&\mbox{Any choice}&&&&\\ \hline
\end{array}\]

\end{example}

We will give an explicit negative answer to the question of Auslander and Reiten by giving a class of contracted preprojective algebras $A$ of Dynkin type $A_n$ that are Cohen-Macaulay with dualising module $W$ with $\idim W=2$ such that $\Omega^2(\mod A)=\CM A$, and such that $A$ is not Iwanaga-Gorenstein.
The following observation shows that such $A$'s form a large family.

\begin{proposition}\label{family of counter example}
Let $\Pi$ be a preprojective algebra of Dynkin type, $J\subset\Delta_0$ a non-empty subset, and $A:=\Pi(\Delta,J)$. Then $A$ satisfies $(\idim A,\fidim A,\domdim A)=(\infty,2,2)$ if and only if $J$ is obtained by the following three steps.
\begin{enumerate}[\rm(i)]
\item Take a non-empty subset $K\subset\Delta_0$ satisfying $\iota(K)=K$ and $(\Delta,K)\neq(A_{2n-1},\{n\})$.
\item Take a non-empty subset $L\subset\Delta_0\setminus K$ such that, for each connected component $C$ of the graph $\Delta\setminus K$, at least one of $C\cap L$ or $\iota(C)\cap L$ is empty.
\item Let $J:=K\sqcup L$.
\end{enumerate}
In this case, $J_\f=K$ and $J_\m=L$ hold.
\end{proposition}

\begin{proof}
We prove the ``if'' part. By construction, $J_\f=K$ and $J_\m=L$ hold. Thus $\iota(J_\f)=\iota(K)=K=J_\f$ holds. Since $L\neq\emptyset$, $\fidim A=2$ holds by Theorem \ref{main ePie 1}(2). Since $\iota(J)\neq J$, $\domdim A=2$ holds by Theorem \ref{main ePie 1}(3). Since $(\Delta,K)\neq(A_{2n-1},\{n\})$ and $\#J\ge2$ by construction, $J$ is impartial. Thus $\idim A=\infty$ holds by Theorem \ref{dimension for type A}.

The ``only if'' part follows from a similar argument. The details is left to the reader.
\end{proof}

We give a single explicit example.

\begin{example}
Let $\Delta=A_n$ with $n\ge4$, $J:=\{1,2,\ldots,n-1\}\subset\Delta_0$ and $A:=\Pi(\Delta,J)$.
Then $A=KQ/I$ holds, where $Q$ is given by
\[\begin{tikzcd}[cells={nodes={cyan!90!black,inner sep=1pt,outer sep=3pt}}]
	|[black]|1 && 2 && 3 && |[text=black]|\cdots && {n-2} && {n-1}
	\arrow["{a_1}", shift left=1, from=1-1, to=1-3]
	\arrow["{b_1}", shift left=1, from=1-3, to=1-1]
	\arrow["{a_2}", shift left=1, from=1-3, to=1-5]
	\arrow["{b_2}", shift left=1, from=1-5, to=1-3]
	\arrow["{a_3}", shift left=1, from=1-5, to=1-7]
	\arrow["{b_3}", shift left=1, from=1-7, to=1-5]
	\arrow["{a_{n-3}}", shift left=1, from=1-7, to=1-9]
	\arrow["{b_{n-3}}", shift left=1, from=1-9, to=1-7]
	\arrow["{a_{n-2}}", shift left=1, from=1-9, to=1-11]
	\arrow["{b_{n-2}}", shift left=1, from=1-11, to=1-9]
\end{tikzcd}\]
and $I=\langle a_1b_1, (b_{n-2}a_{n-2})^2, b_i a_i-a_{i+1} b_{i+1} \mid i=1,...,n-3 \rangle$.
Note that only $1\in J$ is mutable, and so indecomposable projective $A$-modules $P_i=e_i A$ are injective, except for $i=1$. 
By our results $A$ is Cohen-Macaulay of dominant dimension 2 with dualising module $W=\Omega^2(I_1)\oplus I_2\oplus I_3\oplus \cdots I_{n-1}$ for $I_i:=D(e_iA)$, and $A$ is not Iwanaga-Gorenstein.
\end{example}

We give an example of a contracted preprojective algebra of Dynkin type $E_6$ with dominant dimension zero and finitistic dimension two by explicit quiver and relations,which shows that in general we do not have $\CM A= \Omega^2(\mod A)$ for a contracted preprojective algebra of Dynkin type with finitistic dimension two.

\begin{example}
Let $\Delta=E_6$, $J={\def\arraystretch{0.5}\begin{array}{c}\ee\\
\mm\ff\ee\ee\ee\end{array}}$ and $A=\Pi(\Delta,J)$.
Then $A=KQ/I$ holds, where $Q$ is given by
\[\begin{tikzcd}
1 && |[cyan!90!black]|2 
\arrow["{a}", shift left=1, from=1-1, to=1-3]
\arrow["{b}", shift left=1, from=1-3, to=1-1]
\arrow["{c}", from=1-3, to=1-3, loop, in=325, out=35, distance=10mm]
\end{tikzcd}
\]
and $ I=\langle ab, cbac, c^2+bacba\rangle$. 
The dualizing module of this algebra $C$ is given by $W=I_1 \oplus X$, where $I_1$ is the first indecomposable injective $C$-module and the module $X$ is the indecomposable module with dimension vector $[2,3]$ and quiver representation given by 
$a\mapsto \left[\begin{smallmatrix}
1 & 0 & 0 \\ 
0 & 1 & 0
 \end{smallmatrix}\right],\ b\mapsto \left[\begin{smallmatrix}
 0 & 0 \\ 
0 & 0 \\
 0 & 1
 \end{smallmatrix}\right],\ c\mapsto\left[\begin{smallmatrix}
0 & 0 & 0 \\ 
1 & 0 & 0 \\
0 & 0 & 0 
 \end{smallmatrix}\right]$.
In this example $W$ is not in $\Omega^2(\mod A)$ and thus we have that $\CM A \neq \Omega^2(\mod A)$, as $W \in \CM A$ but $W \notin \Omega^2(\mod A)$.
\end{example}

\subsection{Proof of Theorem \ref{dimension for type A}}

To prove Theorem \ref{dimension for type A}(1), we recall the following basic fact. 

\begin{lemma}\label{eBe selfinjective}
Let $A$ be a selfinjective algebra and $e$ an idempotent.
If $\nu(eA)\simeq eA$, then $eAe$ is selfinjective.
\end{lemma}

\begin{proof}
We have isomorphisms $eAe=\Hom_A(eA,eA)\simeq D\Hom_A(eA,\nu(eA))\simeq D\Hom_A(eA,eA)\simeq D(eAe)$ of $eAe$-modules.
\end{proof}

Immediately, we obtain the following observation, which gives a proof of Theorem \ref{dimension for type A}(1).

\begin{proposition}\label{weakly symmetric}
Each contracted preprojective algebra of type $D_n$ with even $n$ and $E_n$ with $n=7,8$ is selfinjective.
\end{proposition}

\begin{proof}
Since $\iota$ is the identity in these cases, the assertion follows from Lemma \ref{eBe selfinjective}.\end{proof}

Next we prove Theorem \ref{dimension for type A}(2) for type $A_n$.
The following observation is crucial.

\begin{proposition}\label{converse yoneda}
Assume that $J\subset[1,n]$ is impartial. For each $i\in[1,n]$, the following assertions hold.
\begin{enumerate}[\rm(1)]
\item $\T(M,X_i)\in\proj A$ if and only if $i\in J$.
\item $\T(M,X_i)\in\inj A$ if and only if $i\in\iota(J)$.
\end{enumerate}
\end{proposition}

To prove this, we need easy observations on preprojective algebras.
We call a path $p$ in the double quiver type of type $A_n$ \emph{minimal} if all arrows in $p$ are either in the right direction, or in the left direction.

\begin{lemma}\label{paths of PPA}
Let $\Pi$ be a preprojective algebra of type $A_n$. Then the following assertions hold.
\begin{enumerate}[\rm(1)]
\item All minimal paths are non-zero elements in $\Pi$.
\item Let $a,b\in[1,n]$ with $a\le b$, and let $p:a\to b$ and $q:b\to a$ be the minimal paths. Then the path $qp:a\to a$ is non-zero in $\Pi$ if and only if $2a\ge b+1$ holds.
\end{enumerate}
\end{lemma}

\begin{proof}[Proof of Proposition \ref{converse yoneda}]
(1) It suffices to prove the ``only if'' part.  Assume that $i\notin J$ satisfies $\T(M,X_i)\in\proj A$.

First, we consider the case that both $[1,i)\cap J$ and $(i,n]\cap J$ is non-empty. Then there exists a minimal right $\M$-approximation $(p\cdot\ ,q\cdot\ ):X_a\oplus X_b\to X_i$, where $a:=\max([1,i)\cap J)$, $b:=\min((i,n]\cap J)$ and $p:a\to i$ and $q:b\to i$ are the minimal paths. Then there exists a triangle
\[X_{i'}\xrightarrow{{r\cdot\choose -s\cdot}} X_a\oplus X_b\xrightarrow{(p\cdot\ q\cdot)} X_i\to X_{i'}[1]\]
where $i':=a+b-i$, and $r:i'\to a$ and $s:i'\to b$ are the minimal paths. Applying $\T(M,-)$, we obtain an exact sequence
\[\T(M,X_{i'})\to\T(M,X_a\oplus X_b)\to\T(M,X_i)\to0,\]
where the right map is the projective cover of $\T(M,X_i)$ and hence an isomorphism. Thus the map ${r\cdot\choose -s\cdot}:\T(M,X_{i'})\to\T(M,X_a\oplus X_b)$ is zero.
This is impossible since the minimal path $t:b\to i'$ gives a morphism $t\cdot:X_b\to X_{i'}$ such that  $rt:b\to a$ is the minimal path and hence gives a non-zero morphism $rt\cdot:X_b\to X_a$ by Lemma \ref{paths of PPA}(1).

Next, we consider the case that at least one of $[1,i)\cap J$ and $(i,n]\cap J$ is empty. 
Without loss of generality, we assume $(i,n]\cap J=\emptyset$, or equivalently, $J\subset [1,i)$.
Then there exists a minimal right $\M$-approximation $p\cdot:X_a\to X_i$, where $a:=\max([1,i)\cap J)$ and $p:a\to i$ is the path of minimal length. Then there exists a triangle
\[X_{i'}\xrightarrow{r\cdot} X_a\xrightarrow{p\cdot} X_i\to X_{i'}[1]\]
where $i':=a+n+1-i$, and $r:i'\to a$ is the minimal path. Applying $\T(M,-)$, we obtain an exact sequence
\[\T(M,X_{i'})\to\T(M,X_a)\to\T(M,X_i)\to0,\]
where the right map is the projective cover of $\T(M,X_i)$ and hence an isomorphism. Thus the map $r\cdot:\T(M,X_{i'})\to\T(M,X_a)$ is zero. But this is impossible. In fact, since $J$ is impartial, $\frac{n+1}{2}\le a<i$ holds. Thus $n+2\le a+i$ and $i'+1=a+n+2-i\le 2a$ hold. By Lemma \ref{paths of PPA}(2), for the minimal path $t:a\to i'$, the composition $rt:a\to a$ is non-zero, and so is the morphism $rt\cdot:X_a\to X_a$. 

(2) By the dual of (1), for $j\in[1,n]$, $\T(X_j,M)\in\proj A^{\op}$ holds if and only if $j\in J$ holds.
By Serre duality, we have $\T(M,X_i)\simeq D\T(X_{\iota(i)},M)$. Thus $\T(M,X_i)\in\inj A$ if and only if $\T(X_{\iota(i)},M)\in\proj A^{\op}$ if and only if $\iota(i)\in J$.
\end{proof}

\begin{lemma}\label{basic of J_f}
Let $Q$ be a Dynkin quiver and $J\subset Q_0$.
\begin{enumerate}[\rm(1)]
\item $J\cap\iota(J)\subset J_\f$ holds.
\item $J_\f\subset\iota(J)$ if and only if $J_\f=J\cap\iota(J)$ if and only if $J_\f=\iota(J_\f)$. 
\end{enumerate}
\end{lemma}

\begin{proof}
(1) This is clear from the definition of $J_\f$.

(2) If $J_\f\subset\iota(J)$, then $J_\f\subset J\cap\iota(J)$ holds. Thus the equality holds by (1).
If $J_\f=J\cap\iota(J)$, then $\iota(J_\f)=\iota(J)\cap J=J_\f$ holds.
If $J_\f=\iota(J_\f)$, then $J_\f\subset\iota(J)$ clearly.
\end{proof}

Now we are able to prove Theorem \ref{dimension for type A}(2) for type $A_n$.

\begin{proof}[Proof of Theorem \ref{dimension for type A}(2)]
(3) $\domdim A\ge2$ if and only if $I_\F\in\proj A$ by \eqref{UIII sequence G} and Lemma \ref{Miyachilemma}. By Proposition \ref{converse yoneda}, this is equivalent to $J_\f\subset\iota(J)$, which is equivalent to $J_\f=\iota(J_\f)$ by Lemma \ref{basic of J_f}(2).

(2) If $J=\iota(J)$, then $A$ is selfinjective. If $J\subset[1,m]$ holds for $m:=\frac{n+1}{2}$, then $A\simeq\End_{k[x]/(x^m)}(\bigoplus_{i\in J}k[x]/(x^i))$ satisfies $\idim A\ge2$.

In the rest, assume $\idim A<\infty$. Since $U$ is a dualizing module, $U\in\proj A$ holds. In particular, $I_\F\in\proj A$ holds, and hence $J_\f=\iota(J_\f)$ holds by (3). Moreover, for each $i\in J_\m$, the sequence \eqref{IIUP sequence G} shows that $\top U_i$ has $S_i$ as a direct summand. Since $U_i$ is indecomposable, $U_i\simeq P_i$. holds and we have an exact sequence
\[0\to\overline{I}_i\to I_i\to P_i\to\overline{P}_i\to0.\]
Let $L_i$ be the image of the middle map.

Now we claim that, for each $i\in J_\m$, $J_\f$ is contained in either $[1,i)$ or $(i,n]$. Otherwise, let $i_-:=\max([1,i)\cap J_\f)$ and $i_+:=\min((i,n]\cap J_\f)$. Since $L_i$ is the kernel of $P_i\to\overline{P}_i$, we have $\top L_i=S_{i_-}\oplus S_{i_+}$.
Since $\max([1,\iota(i))\cap J_\f)=\iota(i_+)$ and $\min((\iota(i),n]\cap J_\f)=\iota(i_-)$ hold, $\top L_i=\top(I_i/\overline{I}_i)=S_{\iota(i_+)}\oplus S_{\iota(i_-)}$.
Comparing the two descriptions of $\top L_i$, we obtain $\{i_-,i_+\}=\{\iota(i_-),\iota(i_+)\}$. Comparing the smaller element, we have $i_-=\iota(i_+)$.
This implies $i_-<\iota(i)<i_+$ and hence $[i,\iota(i)]\cap J_\f\neq\emptyset$, a contradiction to Lemma \ref{enough frozen}. Thus the claim holds.

Without loss of generality, we may assume $i:=\max J_\m<i_+:=\min J_\f$. Since $L_i$ is the kernel of $P_i\to\overline{P}_i$, we have $\top L_i=S_{i_+}$. Since $\max([1,\iota(i))\cap J_\f)=\iota(i_+)$ and $(\iota(i),n]\cap J_\f)=\emptyset$ hold, $\top L_i=\top(I_i/\overline{I}_i)=S_{\iota(i_+)}$. Comparing two descriptions of $\top L_i$, we obtain $i_+=\iota(i_+)$. Since $J_\f=\iota(J_\f)$, we have $J_\f=\{i_+\}$.
Consequently, we have $J\subset[1,\frac{n+1}{2}]$.

(1) This is an easy consequence of (2).
\end{proof}





To prove Theorem \ref{dimension for type A}(3) for type $D_n$ with odd $n$, we enumerated the arrows as follows.
\[\begin{tikzcd}
	&& 2 \\
	\\
	1 && 3 && 4 && \cdots && {n-1} && n
	\arrow["{a_2}", shift left=1, from=1-3, to=3-3]
	\arrow["{a_1}", shift left=1, from=3-1, to=3-3]
	\arrow["{b_2}", shift left=1, from=3-3, to=1-3]
	\arrow["{b_1}", shift left=1, from=3-3, to=3-1]
	\arrow["{a_3}", shift left=1, from=3-3, to=3-5]
	\arrow["{b_3}", shift left=1, from=3-5, to=3-3]
	\arrow["{a_4}", shift left=1, from=3-5, to=3-7]
	\arrow["{b_4}", shift left=1, from=3-7, to=3-5]
	\arrow["{a_{n-2}}", shift left=1, from=3-7, to=3-9]
	\arrow["{b_{n-2}}", shift left=1, from=3-9, to=3-7]
	\arrow["{a_{n-1}}", shift left=1, from=3-9, to=3-11]
	\arrow["{b_{n-1}}", shift left=1, from=3-11, to=3-9]
\end{tikzcd}\]

\begin{proof}[Proof of Theorem \ref{dimension for type A}(3)]

Consider type $D_n$ with odd $n$.
\begin{enumerate}
\item If $1,2\in J$ or $1,2\notin J$, then $A$ is selfinjective.
\item Assume $1\in J$ and $2\notin J$. If $3\in J$, then $J=J_\f$ and $A$ has finitistic dimension $0$.
If $3\notin J$, then $J_\f=J\setminus\{1\}$ and $A$ has finitistic dimension $2$.
\end{enumerate}

Next we consider type $D_n$ in order to verify the statements in the Table \ref{tableexample}.
Let $\Pi=\Pi(D_n)$ denote the preprojective algebra of Dynkin type $D_n$.
If $n$ is even then the algebra is symmetric and thus $e \Pi e$ is selfinjective for every idempotent $e$. We can thus assume that $n$ is odd in the following. 
We deal with the case of primitive idempotents first:

$e \Pi e$ in type $D_n$ is selfinjective if $e$ is primitive and corresponds to a point $i$ with $i>2$ by Lemma \ref{eBe selfinjective}.
It is elementary to see that $e_1 \Pi e_1 \simeq e_2 \Pi e_2$ is isomorphic to $K[x]/(x^{\frac{n-1}{2}})$ and thus selfinjective. We can thus assume now that $e$ is not primitive in the following.
Since the Nakayama permutations swaps only 1 and 2 and fixes all other points, if an idempotent $e$ does not contain $e_1$ and $e_2$, then $e \Pi e$ is selfinjective.
If it contains $e_1$ and $e_2$ then $e \Pi e$ is also selfinjective.
Thus we can assume by symmetry that $e$ contains $e_1$ but not $e_2$ and also that $e$ is the sum of at least two primitive idempotents.
We will consider the following two cases:

\underline{Case 1: $e e_3=0$}
By Theorem \ref{main ePie 1}, we have $A$ non-selfinjective with $\fidim A^{\op}=0$.
This implies that $\domdim A=0$; otherwise, $\domdim A\ge 1$ implies that $\pdim\Omega^{-1}(A)=1$, a contradiction.

\underline{Case 2: $e e_3 \neq 0$}
By Theorem \ref{main ePie 1}, we have $\domdim A=2$.
We now show that the algebra is not Iwanaga-Gorenstein.
Suppose on the contrary that $e\Pi e$ is Iwanaga-Gorenstein. 
Note that, since $e_2$ is not a summand, we have $I_1 = D(e\Pi e_1)\simeq e_2 \Pi e$, the unique indecomposable non-projective injective $A$-module and $P_1=e_1\Pi e$ is the unique indecomposable non-injective projective $A$-module.
Using that $\domdim A=2$
, we have an exact sequence of $A$-modules
\[
0\to P_1 \to P \to P' \to I_1 \to 0
\]
for some projective-injective $P,P' \in \proj A$.

To find $P'$, we consider the short exact sequence of $\Pi$-modules:
$$0 \rightarrow b_2 \Pi \rightarrow e_3 \Pi \xrightarrow{{a_2}\cdot -} e_2\Pi \rightarrow \top e_2\Pi\rightarrow0.$$
Applying the exact functor $(-)e$, we obtain the exact sequence of $A$-modules:
$$0 \rightarrow b_2 \Pi e \rightarrow e_3 \Pi e \rightarrow e_2 \Pi e \rightarrow 0.$$
This means that $P'=e_3 \Pi e$ and $\Omega_{A}^1(I_1)= b_2 \Pi e$.

On the other hand, we have the following exact sequence of $A$-modules:
$$0 \rightarrow e_1 \Pi e \xrightarrow{{b_1}\cdot-} e_3 \Pi e \rightarrow e_3 \Pi e / b_1 \Pi e \rightarrow 0,$$
which means that $P=e_3 \Pi e$ with $\Omega^{-1}_A(P_1)=e_3 \Pi e / b_1 \Pi e$.
Thus, we have the following isomorphism of $A$-modules
\[\Omega_{e \Pi e}^{-1}(P_1)=e_3 \Pi e / b_1 \Pi e \simeq b_2 \Pi e= \Omega_{e \Pi e}^1(I_1).\]
But the isomorphism $e_3 \Pi e / b_1 \Pi e \simeq b_2 \Pi e$ can not hold since $\rad( e_3 \Pi e / b_1 \Pi e )$ has a simple top $S_4$ (in fact, $e_3\Pi e$ is uniserial with top $S_3$ and socle $S_n$), while $\rad(b_2 \Pi e)$ has top given by $S_1\oplus S_4$.  Thus, $A=e\Pi e$ can not be Iwanaga-Gorenstein.
\end{proof}

Now we complete our proof of Theorem \ref{dimension for type A}.

\begin{proof}[Proof of Theorem \ref{dimension for type A}(4)]
The statements have been verified with the computer algebra system \cite{QPA}.
\end{proof}



\begin{thebibliography}{Gus}
\bibitem[AIR]{AIR} Amiot, C.; Iyama, O.; Reiten, I.: {\it Stable categories of Cohen-Macaulay modules and cluster categories.} Amer. J. Math. 137 (2015), no. 3, 813--857.

\bibitem[Ar1]{Ar1} Arnold, V.I.: {\it Normal forms of functions near degenerate critical points, the Weyl groups $\mathcal{A}_k, \mathcal{D}_k, \mathcal{E}_k$ and Lagrangian singularities.} Funkc. Anal. Appl. 6, 1972, 3-25.

\bibitem[Ar2]{Ar2} Arnold, V.I.: {\it Critical points of smooth functions.} In Proceedings International Congress Mathematicians, Vancouver 1974, Vol. 1, Canad. Math. Congress, Montreal, 1975, 19-39.

\bibitem[ASS]{ASS} Assem, I.; Simson, D.; Skowronski, A.: {\it Elements of the representation theory of associative algebras. Vol. 1. Techniques of representation theory.} London Mathematical Society Student Texts, 65. Cambridge University Press, Cambridge, 2006.


\bibitem[Aus1]{Aus1} Auslander, M: \emph{Functors and morphisms determined by objects},  Representation theory of algebras (Proc. Conf., Temple Univ., Philadelphia, Pa., 1976),  pp. 1--244. Lecture Notes in Pure Appl. Math., Vol. 37, Dekker, New York, 1978.

\bibitem[Aus2]{Aus86} Auslander, M.: \emph{Rational singularities and almost split sequences}, Trans. Amer. Math. Soc. 293 (1986), no. 2, 511--531.

\bibitem[APT]{APT} Auslander, M.; Platzeck, M.; Todorov, G.: {\it Homological theory of idempotent Ideals.} Transactions of the American Mathematical Society, Volume 332, Number 2 , August 1992.

\bibitem[AR1]{AR} Auslander, M.; Reiten, I.: {\it Applications of contravariantly finite subcategories.} Adv. Math. 86 (1991), no. 1, 111-152.

\bibitem[AR2]{AR2} Auslander, M.; Reiten, I.: {\it Cohen-Macaulay and Gorenstein Artin algebras.} Representation theory of finite groups and finite-dimensional algebras (Bielefeld, 1991), 221-245, Progr. Math., 95, Birkh\"auser, Basel, 1991. 

\bibitem[AR3]{AR3} Auslander, M.; Reiten, I.: {\it k-Gorenstein algebras and syzygy modules.} Journal of Pure and Applied Algebra Volume 92, Issue 1, 18 February 1994, Pages 1-27.

\bibitem[AR4]{AR4} Auslander, M.; Reiten, I.:  {\it Idun Syzygy modules for Noetherian rings.} J. Algebra 183 (1996), no. 1, 167--185.

\bibitem[ARS]{ARS} Auslander, M.; Reiten, I.; Smalo, S.: {\it Representation Theory of Artin Algebras} Cambridge Studies in Advanced Mathematics, 36. Cambridge University Press, Cambridge, 1997. xiv+425 pp.




\bibitem[AF]{AF} Avramov, L.; Foxby, H.: {\it Ring homomorphisms and finite Gorenstein dimension.}
Proc. London Math. Soc. (3) 75 (1997), no. 2, 241--270.

\bibitem[BFS]{BFS} Bahlekeh, A.; Fotouhi, F. S.; Salarian, S.: {\it Representation-theoretic properties of balanced big Cohen-Macaulay modules.} 
Math. Z. 293, No. 3-4, 1673-1709 (2019).

\bibitem[BST]{BST} Bahlekeh, A.; Salarian, S.; Toosi, Z.: {\it Cohen-Macaulay Noetherian algebras.} Kyoto J. Math. 63, No. 1, 1-22 (2023).



\bibitem[B]{B} Beligiannis, A.: {\it On the relative homology of cleft extensions of rings and abelian categories.} J. Pure Appl. Algebra 150, No. 3, 237-299 (2000).

\bibitem[BR]{BR} Beligiannis, A.; Reiten,I.: {\it Homological and homotopical aspects of torsion theories.} Mem. Am. Math. Soc. 883, 207 p. (2007). 


\bibitem[BH]{BH} Bruns, W., Herzog, J.: {\it Cohen-Macaulay rings.} Cambridge Stud. Adv. Math., 39 Cambridge University Press, Cambridge, 1993.


\bibitem[BGS]{BGS} Buchweitz, R.-O.; Greuel, G.M.; Schreyer, F.O.: {\it Cohen-Macaulay modules on hypersurface singularities II.} Invent. Math. 88, 1987, 165-182.


\bibitem[CB]{Cra00} Crawley-Boevey, W.: \emph{On the exceptional fibres of Kleinian singularities}, Amer. J. Math. 122 (2000), no. 5, 1027--1037.

\bibitem[DH]{DH} Dugas, A.; Huisgen-Zimmerman, B.: {\it Strongly tilting truncated path algebras.} Manuscripta Math. 134 (2011), no. 1-2, 225-257.


\bibitem[E]{E} Eisenbud, D.: {\it Homological algebra on a complete intersection, with an application to group representations}. Trans. Amer. Math. Soc. 260 (1980), no. 1, 35--64.

\bibitem[EJ]{EJ} Enochs, E. E.; Jenda, O. M. G.: {\it Relative homological algebra. Volume 1}. Second revised and extended edition. De Gruyter Expositions in Mathematics, 30. Walter de Gruyter GmbH \& Co. KG, Berlin, 2011.





\bibitem[FGR]{FGR} Fossum, R.; Griffith, P.; Reiten, I.: {\it Trivial extensions of Abelian categories. Homological algebra of trivial extensions of Abelian categories with applications to ring theory. } Lecture Notes in Mathematics. 456. Berlin-Heidelberg-New York: Springer-Verlag. XI, 122 p. (1975).

\bibitem[GL]{GL91} Geigle, W.;Lenzing, H.: \emph{Perpendicular categories with applications to representations and sheaves}, J. Algebra 144 (1991), no. 2, 273--343.

\bibitem[GLS]{GLS13} Geiss, C.; Leclerc, B.; Schr\"oer, J.: \emph{Cluster algebras in algebraic Lie theory}, Transform. Groups 18 (2013), no. 1, 149--178. 

\bibitem[GN]{GN} Goto, S.; Nishida, K.: {\it Finite modules of finite injective dimension over a Noetherian algebra.} J. London Math. Soc. (2) 63 (2001), no. 2, 319--335.

\bibitem[G]{G} Grant, J.: {\it The Nakayama automorphism of a self-injective preprojective algebra.} Bull. Lond. Math. Soc. 52, No. 1, 137-152 (2020).


\bibitem[HU1]{HU1} Happel, D.; Unger, L.: {\it Modules of finite projective dimension and cocovers.} Math. Ann. 306 (1996), no. 3, 445--457.

\bibitem[HU2]{HU2} Happel, D.; Unger, L.: {\it On a partial order of tilting modules.} Algebr. Represent. Theory 8 (2005), no. 2, 147--156.

\bibitem[Ha]{H} Hartshorne, R.: {\it Residues and duality.} Lecture notes of a seminar on the work of A. Grothendieck, given at Harvard 1963/64. With an appendix by P. Deligne. Lecture Notes in Mathematics, No. 20. Springer-Verlag, Berlin-New York, 1966.

\bibitem[Hu]{Hu} Huang, Z.: {\it Syzygy modules for quasi k-Gorenstein rings.} J. Algebra 299 (2006), no. 1, 21--32.

\bibitem[HI]{HI} Huang, Z.; Iyama, O.: {\it Auslander-type conditions and cotorsion pairs.} J. Algebra 318 (2007), no. 1, 93--100.

\bibitem[I]{I} Iyama, O.: {\it Higher-dimensional Auslander-Reiten theory on maximal orthogonal subcategories.} Adv. Math. 210 (2007), no. 1, 22--50.

\bibitem[IyaSol]{IyaSol} Iyama, O; Solberg, {{\O}}.: {\it Auslander-Gorenstein algebras and precluster tilting.} Advances in Mathematics Volume 326, 21 February 2018, Pages 200-240.

\bibitem[IW]{IW} Iyama, O.; Weymss, M.: {\it Tits Cone Intersections and Applications.} \url{https://www.maths.gla.ac.uk/~mwemyss/MainFile_for_web.pdf}.

\bibitem[IY]{IY} Iyama, O.; Yoshino, Y.: {\it Mutation in triangulated categories and rigid Cohen-Macaulay modules.} Invent. Math. 172 (2008), no. 1, 117--168.

\bibitem[IZ]{IZ} Iyama, O.; Zhang, X.: {\it Tilting modules over Auslander-Gorenstein algebras.} Pac. J. Math. 298, No. 2, 399-416 (2019).

\bibitem[IK]{IK} Iyengar, S.; Krause, H.: {\it Acyclicity versus total acyclicity for complexes over Noetherian rings.} Doc. Math. 11 (2006), 207--240.

\bibitem[J]{J} J\o rgensen, P.: {\it Quotients of cluster categories.} Proc. Roy. Soc. Edinburgh Sect. A 140 (2010), no. 1, 65--81.

\bibitem[KS]{KS97} Kashiwara, M.; Saito, Y.: \emph{Geometric construction of crystal bases}, Duke Math. J. 89 (1997), no. 1, 9--36.


\bibitem[K]{K} Kn\"orrer, H.: {\it Cohen-Macaulay modules on hypersurface singularities I.} Invent. Math. 88, 1987, 153-164.


\bibitem[LW]{LW} Leuschke, G.; Wiegand, R.:{\it Cohen-Macaulay representations.} Mathematical Surveys and Monographs, 181. American Mathematical Society, Providence, RI, 2012.

\bibitem[L]{Lus91} Lusztig, G.: \emph{Quivers, perverse sheaves, and quantized enveloping algebras}, J. Amer. Math. Soc. 4 (1991), no. 2, 365--421.

\bibitem[MarVil]{MarVil} Martinez Villa, R.: {\it Modules of dominant and codominant dimension.} Communications in algebra, 20(12), 3515-3540, 1992.

\bibitem[M]{M} Miyachi, J.: {\it Injective resolutions of noetherian rings and cogenerators.} Proc. Am. Math. Soc. 128, No. 8, 2233-2242 (2000). 


\bibitem[Na]{Nak94} Nakajima, H.:  \emph{Instantons on ALE spaces, quiver varieties, and Kac-Moody algebras}, Duke Math. J. 76 (1994), no. 2, 365--416.

\bibitem[Ni]{Ni} Nishida, K.: {\it Dualizing modules for orders and Artin algebras.} Algebr. Represent. Theory 5 (2002), no. 2, 137--147.

\bibitem[PS]{PS} Peskine, C.; Szpiro, L.: {\it Dimension projective finie et cohomologie locale. Applications a la demonstration de conjectures de M. Auslander, H. Bass et A. Grothendieck.} Publ. Math., Inst. Hautes Etud. Sci. 42, 47-119 (1972).

\bibitem[QPA]{QPA} The QPA-team, QPA - Quivers, path algebras and representations - a GAP package, Version 1.33; 2022 \url{https://folk.ntnu.no/oyvinso/QPA/}.


\bibitem[R]{R} Roberts, P.: {\it Le theoreme d'intersection.} C. R. Acad. Sci., Paris, Ser. I 304, 177-180 (1987).

\bibitem[Sh]{Sh} Sharp, R. Y.: {\it Finitely generated modules of finite injective dimension over certain Cohen-Macaulay rings.} Proc. London Math. Soc. (3) 25 (1972), 303--328.

\bibitem[SkoYam]{SkoYam} Skowro\'{n}ski, A.; Yamagata, K.: {\it Frobenius Algebras I: Basic Representation Theory.} EMS Textbooks in Mathematics, 2011.

\bibitem[Sk]{Sko} Skowro\'{n}ski, A.: {\it Periodicity in representation theory of algebras.} \url{https://webusers.imj-prg.fr/~bernhard.keller/ictp2006/lecturenotes/skowronski.pdf}







\bibitem[Y]{Y} Yoshino, Y.: {\it Cohen-Macaulay modules over Cohen-Macaulay rings.} London Mathematical Society Lecture Note Series. 146. Cambridge (UK): Cambridge University Press. 177 p. (1990).


\end{thebibliography}
\end{document}